\documentclass[10pt]{amsart}
\usepackage[utf8]{inputenc}
\usepackage{graphicx,amssymb,amsmath,amsthm,wrapfig}
\usepackage{caption}
\usepackage{subcaption}
\usepackage{indentfirst}
\usepackage{cancel}
\usepackage{lipsum}
\usepackage{epstopdf}
\usepackage{epsfig}
\usepackage[normalem]{ulem}
\usepackage{soul}
\usepackage{comment}
\usepackage{enumerate,color}   
\usepackage[left=1in,right=1in,top=1in,bottom=.8in]{geometry}
\usepackage[overload]{empheq}

\newcommand{\veps}{\varepsilon}
\newcommand{\ve}{\varepsilon}
\newcommand{\U}{\overline{U}}

\newcommand{\dd}{\partial}

\newcommand{\uuu}{\mathfrak{u}}
\newcommand{\wt}{\widetilde}

\numberwithin{equation}{section}
\usepackage{soul}
\everymath{\displaystyle}
\usepackage{todonotes}

\makeatletter

\renewcommand\subsubsection{\@secnumfont}{\bfseries}%
\renewcommand\subsubsection{\@startsection{subsubsection}{3}
  \z@{.5\linespacing\@plus.7\linespacing}{-.5em}%
  {\normalfont\bfseries}}
  
  \makeatother

\theoremstyle{plain}
\newtheorem{theorem}{Theorem}[section]
\newtheorem{lemma}[theorem]{Lemma}
\newtheorem{proposition}[theorem]{Proposition}

\theoremstyle{definition}

\newtheorem*{defi*}{Definition}
\theoremstyle{remark}
\newtheorem{remark}{Remark}

\theoremstyle{remark}

\usepackage{etoolbox}
\makeatletter
\let\alignts@preamble\align@preamble
\patchcmd{\alignts@preamble}{\displaystyle}{\textstyle}{}{}
\patchcmd{\alignts@preamble}{\displaystyle}{\textstyle}{}{}

\def\alignts{\let\align@preamble\alignts@preamble\start@align\@ne\st@rredfalse\m@ne}

\makeatother

\allowdisplaybreaks

\title[Formation of singularities for the Euler-Poisson system]{Delta-shock for the pressureless Euler-Poisson system}

\author[J. Bae]{Junsik Bae}
\address[JB]{Department of Mathematical Sciences, Ulsan National Institute of Science and Technology, Ulsan, 44919, Korea}
\email{junsikbae@unist.ac.kr}
 
 \author[Y. Kim]{Yunjoo Kim}
\address[YK]{Department of Mathematical Sciences, Ulsan National Institute of Science and Technology, Ulsan, 44919, Korea}
\email{gomuli3@unist.ac.kr}

\author[B. Kwon]{Bongsuk Kwon}
\address[BK]{Department of Mathematical Sciences, Ulsan National Institute of Science and Technology, Ulsan, 44919, Korea}
\email{bkwon@unist.ac.kr}

\date{\today}


\begin{document}

\begin{abstract}
%
%
%
We study singularity formation for the pressureless Euler-Poisson system of cold ion dynamics.
In contrast to the Euler-Poisson system with pressure, when its smooth solutions experience $C^1$ blow-up, the $L^\infty$ norm of the density becomes unbounded, which is often referred to as a delta-shock. We provide a constructive proof of singularity formation to obtain an exact blow-up profile and the detailed asymptotic behavior of the solutions near the blow-up point in both time and space. Our result indicates that at the blow-up time $t=T_\ast$, the density function is unbounded but is locally integrable with the profile of $\rho(x,T_\ast) \sim (x-x_*)^{-2/3}$ near the blow-up point $x=x_\ast$. This profile is not yet a \emph{Dirac measure}. 
On the other hand, the velocity function has $C^{1/3}$ regularity at the blow-up point. Loosely following our analysis, we also obtain an exact blow-up profile for the pressureless Euler equations. 

\noindent{\it Keywords}:
Euler-Poisson system; Boltzmann-Maxwell relation; Cold ion; Singularity; Delta-shock
\end{abstract}

\maketitle 

\section{Introduction}
We consider the pressureless Euler-Poisson system  in a non-dimensional form:
\begin{subequations}\label{EP-p}
	\begin{align}
		& \rho_t +  (\rho u)_x = 0, \label{EP_1-p} 
		\\ 
		& u_t  + u u_x = -  \phi_x, \label{EP_2-p} 
		\\
		& - \phi_{xx} = \rho - e^\phi, \label{EP_3-p} 
	\end{align}
\end{subequations} 
where $\rho>0$, $u$ and $\phi$ are the unknown functions of $(x,t) \in \mathbb{R}\times [-\ve, \infty)$ representing the ion density, the fluid velocity for ions, and  the electric potential, respectively. 
The pressureless Euler-Poisson system \eqref{EP-p} is a fundamental fluid model describing the dynamics of cold ions in an electrostatic plasma. 
In fact, \eqref{EP-p} is an ideal model for cold ions, for which  the ion pressure is neglected due to the assumption that   the ion temperature $T_i$
is much lower than the electron temperature $T_e$. 
We refer to \cite{Ch,Dav} for more detailed physicality of the model.

A question of global existence or finite-time blow-up of smooth solutions to \eqref{EP-p} naturally arises in
the study of the large-time dynamics of the Euler-Poisson system, for instance the stability of solitary waves \cite{HS}.
To our knowledge, no global existence result has been obtained except for the curl-free flow in 3D, see \cite{GP}. 
 In general, it is known that smooth solutions to nonlinear hyperbolic equations fail to exist globally in time when the gradient of initial velocity is negatively
large, \cite{J, L, TL}. 
For \eqref{EP-p}, it is proved that  under suitable conditions on the initial data, smooth solutions may leave the class of  $C^1$ in  finite time,  see   \cite{BCK, CKKT, Liu}. 
In fact, the study of  \cite{BCK, CKKT} proposes some sufficient conditions for the $C^1$ blow-up, that do not require the initial velocity to have a negatively large gradient. 
Due to the nature of the method of characteristics, only limited information about the blow-up solutions to \eqref{EP-p} is obtained, for instance, $
\| \rho(\cdot,t) \|_{L^\infty} \to \infty$  and  $\|u_x(\cdot,t)\|_{L^\infty}\to\infty$ as $t\to T_\ast$, 
%
%
%
%
%
where $T_\ast$ is the life span of the smooth solution. 
This $L^\infty$ norm blow-up for the density function, sometimes referred to as the delta-shock, does not occur in the presence of pressure for the Euler-Poisson system, see \cite{BCK, BKK}. 
For this exotic blow-up behavior, no detailed information, such as the spatial configuration, is obtained. 

In this paper, we provide a constructive proof of  singularity formation to obtain an exact blow-up profile and the asymptotic behavior of the blow-up solutions. 
More specifically, we find that the solution $\rho$ blows up like $\rho(x,t)  \sim (T_\ast-t + (x-x_\ast)^{2/3})^{-1}$ asymptotically near the blow-up point $(x_\ast, T_\ast)$, while the velocity $u(\cdot,t)$ has  $C^{1/3}$ regularity at $t=T_\ast$, i.e., a cusp-type singularity. 
% ... 
To show this, we borrow the ideas of \cite{BSV, BSV2}, which are used to show the singularity formation for the multi-D Euler equations. 
Specifically we establish the global stability estimates in the self-similar variables by utilizing the self-similar blow-up profile for the Burgers equation 
\begin{equation}\label{1D-Burgers}
	\partial_t \uuu + \uuu \partial_x \uuu =0. 
\end{equation}

\subsection{Self-similar blow-up profile of the Burgers equation}
Introducing the self-similar variables
\begin{equation*}\label{change_var-1}
	y\left(x,t\right) := \frac{x}{\left( -t\right)^{3/2}}, \quad s(t) :=-\log\left(  - t \right), 
\end{equation*}
and a new unknown function $\mathfrak{U}(y,s)$ defined by  
\begin{equation*}\label{WZPhi-1}
	\mathfrak{U}(y,s) :=   e^{s/2}\uuu(x,t), 
\end{equation*}
we see from \eqref{1D-Burgers} that $\mathfrak{U}(y,s)$ satisfies
\begin{equation}\label{W-eq-0-1}
	\left(\partial_s-\frac{1}{2}\right)\mathfrak{U}+\left( \frac{3}{2} y+\mathfrak{U}\right)\mathfrak{U}_y =0. 
\end{equation}
It is known, see \cite{EF} for instance, that \eqref{W-eq-0-1} admits a  \emph{stable}  steady solution $\overline{U}(y)$ satisfying 
\begin{equation}\label{Burgers_SS}
	-\frac12 \overline{U}(y) + \frac32 y \overline{U}'(y) + \U (y) \U'(y) =0.
\end{equation}
Multiplying \eqref{Burgers_SS} by $\overline{U}^{-4}$
and integrating the resultant  with respect to $y$,  we find that 
\begin{equation}\label{Burgers_SS'}
	y+\overline{U} + c \overline{U}^3 =0,
\end{equation}
where    $c>0$ is a constant of integration. 
Thanks to the space-amplitude scaling invariance,   we can set $c=1$ (see Section~\ref{scaling-IC} for more details). Then, it is easy to check that      
\begin{equation*}
	y^{-1/3} \overline{U}(y) \to  -1 \quad \text{ and } 	\quad |y^{2/3} \overline{U}'(y)|\rightarrow \frac{1}{3}
	\quad 
	\text{as } |y| \rightarrow \infty, 
\end{equation*}
and 
\begin{equation*}
	\overline{U}(0)=0, \quad \overline{U}'(0)=-1, \quad \overline{U}''(0)=0, \quad \overline{U}'''(0)=6, \quad \overline{U}^{(4)}(0)=0. 
\end{equation*}
In fact, an explicit form of $\overline{U}$ can be obtained by directly solving \eqref{Burgers_SS}. 
However, it is not convenient to use it in our analysis  due to its complicated form. Instead, we make use of the quantitative properties of $\overline{U}$ which can be easily obtained from \eqref{Burgers_SS'}. We derive and give a list of the properties of $\overline{U}$ in Section~\ref{inequalities}.

%

\subsection{Assumptions on the initial data.}\label{Initial_subs}
We describe the initial conditions that lead to the finite-time $C^1$ blow-up. For simplicity, let the initial time $t=-\veps$, where $\veps$ is to be chosen sufficiently small. I.e., we consider the initial value problem for \eqref{EP-p} in $[-\veps,\infty)\times \mathbb{R}$, with the initial data
\begin{equation}\label{in-H-C}
	(\rho_0-1,u_0)(x) := (\rho-1,u)(x,-\veps)\in H^4(\mathbb{R})\times H^5(\mathbb{R}) \subset C^3(\mathbb{R})\times C^4(\mathbb{R}).
\end{equation}

We shall focus on the case in which $\partial_xu_0$ is negatively large and $\partial_xu_0$ attains its non-degenerate global minimum at $x=0$. More specifically, thanks to the structure of \eqref{EP-p}, without loss of generality,  we assume that 
\begin{equation}
	u_0(0)= 0, \quad \partial_xu_0(0)=-\veps^{-1}, \quad \partial_x^2u_0(0)=0, \quad \partial_x^3u_0(0)=6\veps^{-4}. \label{init_w_3-p}
\end{equation}
See Appendix~\ref{scaling-IC} for more details on how the initial conditions can be relaxed to \eqref{init_w_3-p}. 
We remark that  \eqref{init_w_3-p}  ensures that the new variable $U$, which will be defined in \eqref{WZPhi-p}, satisfies $\partial_y^nU(0,s_0)=\overline{U}^{(n)}(0)$ for $n=0,1,2,3$. Here,  $\overline{U}$ is the stable self-similar blow-up profile to the Burgers equation defined in \eqref{Burgers_SS}.

 We also assume that the first four derivatives of $u_0$ have the bounds as  
\begin{equation}\label{init_24-p}
	\|\partial_x u_0\|_{L^{\infty}}\leq \veps^{-1},\quad
	\|\partial_x^2u_0\|_{L^{\infty}}\leq \veps^{-5/2},\quad \|\partial_x^3u_0\|_{L^{\infty}}\leq 7\veps^{-4}, \quad \|\partial_x^4u_0\|_{L^{\infty}}\leq \veps^{-11/2}.
\end{equation}	
Furthermore, we impose the conditions concerning the asymptotic behavior of $\partial_xu_0$ and $\partial_x \rho_0$ 
near $x=0$ and $x=\infty$ as 
\begin{equation} \label{4.3a-p}
	\left|\varepsilon(\partial_x u_0)(x)-\overline{U}'\left(\frac{x}{\varepsilon^{3/2}}\right)\right|\leq \min\left\{\frac{\left(\frac{x}{\varepsilon^{3/2}}\right)^2}{40\left(1+\left(\frac{x}{\varepsilon^{3/2}}\right)^2\right)}, \frac{22}{25\left(8+\left(\frac{x}{\veps^{3/2}}\right)^{2/3}\right)}\right\}, 
\end{equation}
 
\begin{equation} \label{4.3a2-p}
	\limsup_{|x|\rightarrow \infty}  \left| x^{2/3} \partial_xu_0(x) \right| \leq  \frac{1}{2}, 
\end{equation}
and 
\begin{equation}
\veps|\rho_0(x)-1|\leq \frac{A}{	2\left( 8+\left(\frac{x}{\veps^{3/2}}\right)^{2/3} \right) }, \qquad \inf_{x\in\mathbb{R}}\rho_0(x)>0, \label{init_ttP-p}
\end{equation}
where 
\begin{equation}\label{Aeq}
	A:=\min\left\{\frac{1}{8\sup_{y\in\mathbb{R}} I}, \frac{ e^{-1/4}}{(\sup_{y\in\mathbb{R}} I)^2}\right\} \quad \text{with} \quad I(y):=  (y^{2/3} +1) \int_{-\infty}^\infty \frac{e^{-|y-y'|}}{ |y'|^{2/3}}
	dy'.
\end{equation} 
It is easy to check that $\sup_{y\in\mathbb{R}}I(y)>0$ is finite.

We remark that \eqref{4.3a-p}--\eqref{init_ttP-p} are the conditions of localization, which ensure that the singularity develops along the modulation curve where $u$ keeps its steepest gradient. 
In our analysis utilizing the maximum principle for the transport-type equations, the spatial localization of the solutions is essential to obtain global stability estimates. 
We also remark that these conditions allow $(\rho_0-1,u_0)$ to decay fast enough to $(0,0)$ as $|x| \to \infty$ so that it can also belong to appropriate Sobolev spaces $H^k(\mathbb{R})$. This is consistent with \eqref{in-H-C}.


	In what follows, $C^\beta(\Omega)$  denotes  the H\"older space with exponent $\beta$, and $[u]_{C^{\beta}(\Omega) }$ is the associated H\"older norm as 
\begin{equation*}
	[u]_{C^{\beta}(\Omega) } := \sup_{x,y\in \Omega, x\ne y} \frac{|u(x) - u(y) | }{|x-y|^{\beta}}.
\end{equation*} 
%
		\begin{theorem}\label{mainthm1}
	There is a constant $\ve_0>0$
	such that for each $\veps\in(0,\veps_0)$, if the initial data $(\rho_0-1, u_0)\in H^4(\mathbb{R})\times H^5(\mathbb{R})$ satisfies \eqref{in-H-C}--\eqref{init_ttP-p},
	then there is a unique smooth solution 
	$(\rho,u)\in ( C^3\times C^4)  \left([-\veps,T_\ast) \times \mathbb{R}\right)$
	to \eqref{EP-p} whose the maximal existence time $T_\ast$ is finite.  Furthermore, it holds that
	\begin{enumerate}[(i)]
		\item $\sup_{t<T_\ast} \left[ u(\cdot, t) \right]_{C^{\beta}}<\infty \quad \text{for}\quad \beta\leq 1/3$, $\qquad  \lim_{t\rightarrow T_\ast} \left[ u (\cdot, t) \right]_{C^{\beta}}=\infty \quad \text{for}\quad \beta> 1/3$;
		\item for $\beta>1/3$, the temporal blow-up rate is given by 
		\ \begin{equation*} \left[ u(\cdot, t) \right]_{C^\beta}\sim (T_\ast-t)^{-\frac{3\beta-1}{2}}
		\end{equation*} for all $t$ sufficiently close to $T_\ast$\footnote{Here $A(t) \sim B(t) $  means $C^{-1} B(t) \le A(t) \le C B(t)$ for some $C>0$ independent of $t$};
		\item 
the density function $\rho(x,t)$ blows up as   
		\begin{equation}\label{den-profile}
			|\rho(x,t)-1| \sim \frac{1}{(T_\ast-t) + (x-x_*)^{2/3}}
		\end{equation} 
		for all $(x,t)$ sufficiently close to $(x_*,T_\ast)$. Here $x_\ast$ is the  blow-up location, which is given by $x_\ast:= \xi(T_\ast)$ with $\xi$ defined in \eqref{modulation-new}--\eqref{Modul_init-p}. 
	\end{enumerate} 
\end{theorem}
%
Note that $\veps_0>0$ is a constant depending only on the initial data $\rho_0$ and $u_0$.
The proof of Theorem~\ref{mainthm1} is given in Subsection~\ref{thm-pf}. 
The study of \cite{BCK} finds that when $C^1$ blow-up occurs, the density function $\rho$ blows up in the $L^\infty$ norm, while the velocity function $u$ remains bounded. Due to the nature of the characteristic method, only limited information is obtained; for instance, 
\[
\lim_{t \nearrow T_\ast}\sup_{x \in \mathbb{R}}\rho(x,t) = +\infty \quad \text{ and } \quad \sup_{x \in \mathbb{R}} |u_x(x,t)| \sim \frac{1}{T_\ast-t}
\]
for all $t<T_\ast$ sufficiently close to $T_\ast$. 
In order to obtain  more detailed features of the singularity, such as its spatial configuration, we provide a constructive proof of the blow-up.
More specifically, 
 Theorem~\ref{mainthm1} demonstrates that a cusp-type singularity for the velocity $u$ develops from smooth initial data. The smooth solution $u$ exhibits $C^1$ blow-up in finite time, while it still belongs to $C^{1/3}$ at the blow-up time. 
When this occurs, on the other hand, the density function $\rho$ experiences a more exotic type of blow-up, called the delta-shock, for which an exact blow-up profile is obtained as in the assertion $(iii)$. 
See Figure 1 and 2 for the asymptotic density blow-up profiles described by \eqref{den-profile} and the numerical solution $\rho$ that blows up from smooth initial data, respectively. 
We remark that at the blow-up time $t=T_\ast$, the density function is unbounded but is locally integrable near the blow-up point $x_\ast$ with the profile of $\rho(x,T_\ast) \sim (x-x_*)^{-2/3}$. This is not a \emph{Dirac measure} yet. It is an intriguing conjecture that  the density function $\rho$ and the velocity function $u$ evolve into a Dirac measure as a measure-valued solution and a shock wave with a jump discontinuity, respectively, after the blow-up time $T_\ast$. 

As mentioned earlier, the studies of \cite{BCK, CKKT} provide certain sufficient conditions for the initial data that lead to $C^1$ blow-up in finite time, which do not require the velocity to have a large gradient. When $C^1$ blow-up occurs from these data, just a short time prior to the blow-up time, the solutions will satisfy \eqref{in-H-C}--\eqref{init_ttP-p} so that they will blow up as described in Theorem~\ref{mainthm1}. We conjecture that \emph{generic} $C^1$ blow-up solutions to \eqref{EP-p} are of this type.   

 We also remark that further observations of the approach in \cite{BCK} allow us to discover a wider class of initial data that lead to $C^1$ blow-up. We present the blow-up criterion and provide a short proof of the blow-up under the criterion in Appendix~\ref{BC-app} for future reference. This blow-up criterion turns out to be identical to the one given in Theorem 1.11 of \cite{CKKT}, but our observation makes the proof significantly simpler and shorter. 

%
%
%
%
%
%


\begin{figure}[h]
\begin{tabular}{cc}
\resizebox{60mm}{!}{\includegraphics{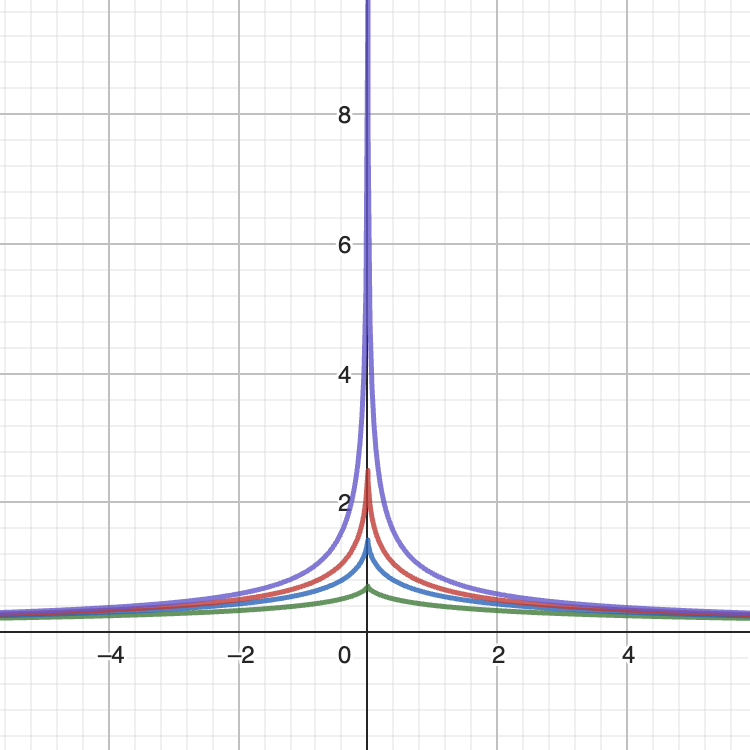}}  & \resizebox{85mm}{!}{\includegraphics{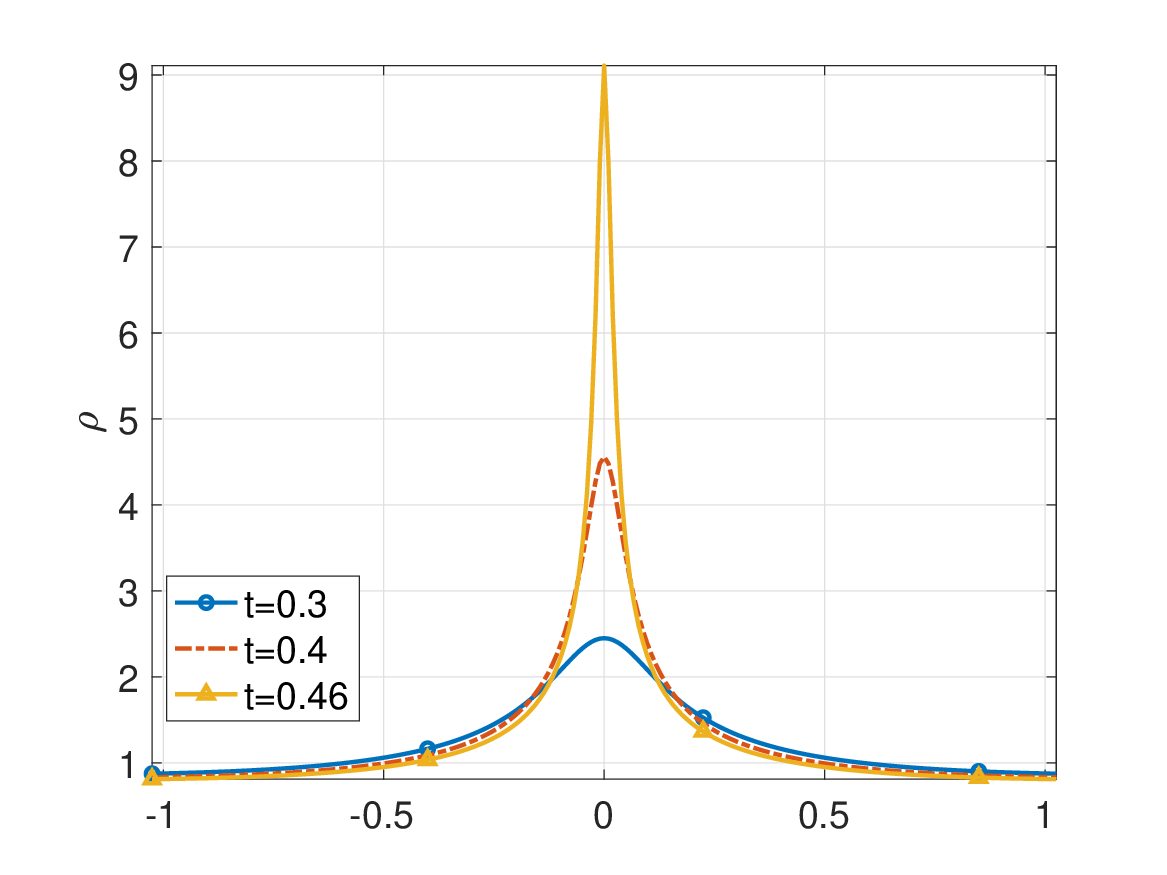}}
 \\
(a)
& (b) 
\end{tabular}
\caption{
  (a) the asymptotic blow-up profile for $\rho$ described by \eqref{den-profile} as $t\to T_\ast$.  (b) the numerical solution $\rho$ that blows up with the initial data  $\rho_0 = 1$ and  $u_0= -\text{sech}(2x)\text{tanh}(2x)$.}
\label{Fig2}
\end{figure}

%

\emph{Outline of the paper:} In Section~\ref{st-est}, using the bootstrap argument, we establish the global stability estimates in the self-similar variables. Using these estimates, we prove Theorem~\ref{mainthm1} at the end of this section. Section~\ref{sec3-boot} is devoted to the closure of the bootstrap assumptions, used to derive the global stability estimates. In Section~\ref{PE-blow-up}, we present a similar blow-up result for the pressureless Euler equations and provide a brief outline the proof.
In the Appendix, we provide some estimates for the Burgers blow-up profile and other miscellaneous materials used throughout our analysis.  

\section{Stability Estimates in the self-similar variables}\label{st-est}
In this section, we establish the global stability estimates in the self-similar time, of which we make use to prove Theorem~\ref{mainthm1}.
First, we rewrite \eqref{EP-p} in the self-similar variables incorporated with dynamic modulation functions.

\subsection{Self-similar variables and modulations.}
We define three dynamic modulation functions $(\tau,\kappa,\xi)(t) : [-\veps,\infty) \rightarrow \mathbb{R}$ satisfying a system of ODEs:
\begin{subequations}\label{modulation-new}
	\begin{align}
		&\dot{\tau}=-(\tau(t)-t)^2\partial_x^2\phi(\xi(t),t),
		\\
		&\dot{\kappa}=-\frac{(\tau(t)-t)^{-1}\partial_x^3\phi(\xi(t),t)}{\partial_x^3u(\xi(t),t)}-\partial_x\phi(\xi(t),t),
		\\
		&\dot{\xi}=\frac{\partial_x^3\phi(\xi(t),t)}{\partial_x^3u(\xi(t),t)}+\kappa(t),
	\end{align}
\end{subequations}
with the initial values
\begin{equation}\label{Modul_init-p}
	\tau(-\veps)=0,  \quad \kappa(-\veps)= 0, \quad \xi(-\veps)=0.
\end{equation}
By a standard ODE theory, the initial value problem for \eqref{modulation-new} with the initial data \eqref{Modul_init-p} admits a unique local $C^1$ solution $(\tau,\kappa,\xi)$.
Here, we define $T_*$ as the first time such that $\tau(t)=t$,  
	i.e., 
	\begin{equation}\label{T-star}
	T_* := \inf\{ t \in [-\ve,\infty) : \tau(t) = t\}.
	\end{equation}
	Later we will show that $T_*$ is finite and is the blow-up time of $\partial_x w$ in the proof of Theorem~\ref{mainthm1}.

We introduce  the following self-similar variables
\begin{equation}\label{change_var}
	y(x,t)= \frac{x-\xi(t)}{(\tau(t)-t)^{3/2}}, \quad s(t)=-\log\left( \tau(t) - t \right).
\end{equation}
Like the self-similar solution of the Burgers equation, we rewrite $(\rho,u,\phi)$ as
\begin{equation}\label{WZPhi-p}
	  \rho(x,t)-1 =e^{s} P(y,s), \quad u(x,t) = e^{-s/2} U(y,s) + \kappa(t), \quad \phi(x,t) = \Phi(y,s).
\end{equation}
From this setting, we can track the blowup time, location and its amplitude of $U$ by determining $\tau$, $\xi$ and $\kappa$. We also see from \eqref{Modul_init-p} and \eqref{change_var} that
\begin{equation*}
	s_0:=s(-\veps)=-\log\veps.
\end{equation*}

Inserting the Ansatz \eqref{WZPhi-p} into the system \eqref{EP-p} and using
\begin{equation}\label{dx-dy}
	\frac{\partial y}{\partial t} = -\dot{\xi}e^{3s/2}+\frac{3}{2}ye^s(1-\dot{\tau}), \qquad \frac{\partial y}{\partial x} = e^{3s/2}, \qquad \frac{\partial s}{\partial t}=(1-\dot{\tau})e^s, 
\end{equation} 
which follow from \eqref{change_var},  we obtain that $P,U$ and $\Phi$ satisfy the equations 
\begin{subequations}\label{EP2-p}
	\begin{align}
		& \partial_sP + \left(1 +\frac{U_y}{1-\dot{\tau}}\right)P + \mathcal{U} P_y = -\frac{e^{-s}U_y}{1-\dot{\tau}}, \label{EP2_1-p} 
		\\ 
		& \partial_s U - \frac{1}{2}U + \mathcal{U} U_y = -\frac{\Phi_y e^s}{1-\dot{\tau}} - \frac{e^{-s/2}\dot{\kappa}}{1-\dot{\tau}}, \label{EP2_2-p} 
		\\
		& -\Phi_{yy}e^{3s} = Pe^s+1 - e^{\Phi}, \label{EP2_3-p} 
	\end{align}
\end{subequations}
where 
\begin{equation}\label{U-p}
	\mathcal{U}:= \frac{U}{1-\dot{\tau}} + \frac{3}{2}y + \frac{e^{s/2}(\kappa-\dot{\xi})}{1-\dot{\tau}} .
\end{equation}
For later use, we present the differentiated equations. Applying $\partial_y^n$ to \eqref{EP2_1-p}, we have
\begin{subequations}\label{EP2_2D-p}
	\begin{align}
		& \left( \partial_s + 1 + \frac{U_y}{1-\dot{\tau}} \right)U_y + \mathcal{U} U_{yy} = -\frac{ e^s\Phi_{yy}}{1-\dot{\tau}},\label{EP2_2D1-p} 
		\\ 
		& \left( \partial_s + \frac{5}{2} + \frac{3U_y}{1-\dot{\tau}} \right)U_{yy} + \mathcal{U} \partial_y^3 U  = -\frac{ e^s\partial_y^3\Phi}{1-\dot{\tau}},  \label{EP2_2D2-p} 
		\\
		& \left( \partial_s + 4 + \frac{4U_y}{1-\dot{\tau}} \right)\partial_y^3U + \mathcal{U} \partial_y^4 U  = - \frac{3U_{yy}^2}{1-\dot{\tau}} -\frac{ e^s\partial_y^4\Phi}{1-\dot{\tau}}, \label{EP2_2D3-p}
		\\
		& \left( \partial_s + \frac{11}{2} + \frac{5U_y}{1-\dot{\tau}} \right)\partial_y^4U + \mathcal{U} \partial_y^5 U  = - \frac{10 U_{yy} \partial_y^3U}{1-\dot{\tau}} -\frac{ e^s\partial_y^5\Phi}{1-\dot{\tau}}. \label{EP2_2D4-p}
	\end{align}
\end{subequations}
Similarly, taking $\partial_y^n$ of \eqref{EP2_2-p}, we obtain
\begin{subequations}\label{EP2_3D-p}
	\begin{align}
		& \left( \partial_s + \frac{5}{2} + \frac{2 U_y}{1-\dot{\tau}} \right)P_y + \mathcal{U} P_{yy} = -\frac{ U_{yy} P}{1-\dot{\tau}}-\frac{e^{- s}}{1-\dot{\tau}}U_{yy}, \label{EP2_3D1-p} 
		\\ 
		& \left( \partial_s + 4 + \frac{3U_y}{1-\dot{\tau}} \right)P_{yy} + \mathcal{U} \partial_y^3 P  = -\frac{ 3U_{yy}P_{y}+\partial_y^3UP}{1-\dot{\tau}}-\frac{e^{- s}}{1-\dot{\tau}}\partial_y^3U,  \label{EP2_3D2-p} 
		\\
		& \left( \partial_s + \frac{11}{2} + \frac{4U_y}{1-\dot{\tau}} \right)\partial_y^3P + \mathcal{U} \partial_y^4 P
  \label{EP2_3D3-p}
   =  - \frac{\partial_y^4UP+4\partial_y^3UP_y + 6U_{yy}P_{yy}}{1-\dot{\tau}}-\frac{e^{- s}}{1-\dot{\tau}}\partial_y^4U. 
	\end{align}
\end{subequations}

We impose the constraints
\begin{equation}\label{constraint-p}
U(0,s) = 0, \quad  U_y(0,s) = -1, \quad U_{yy}(0,s)=0 \qquad \text{for all}\quad s\geq s_0. 
\end{equation}
Substituting \eqref{constraint-p} into \eqref{EP2_1-p}, \eqref{EP2_2D1-p} and \eqref{EP2_2D2-p}, we obtain the following system of ODEs:
\begin{subequations}\label{modulation}
	\begin{align}
		\dot{\tau} &= -e^s\Phi_{yy}(0,s),\label{Eq_Modul1-p} 
		\\
		\dot{\kappa}&=-e^{3s/2}\left(\frac{\partial_y^3\Phi(0,s)}{\partial_y^3U(0,s)}+\Phi_y(0,s)\right),\label{Eq_Modul0-p} 
		\\
		\dot{\xi} &=e^{s/2}\frac{\partial_y^3\Phi(0,s)}{\partial_y^3U(0,s)}+\kappa(t).\label{Eq_Modul2-p}
	\end{align}
\end{subequations}
	We note that \eqref{modulation} is an equivalent system of \eqref{modulation-new}. It is straightforward to check that  
	\begin{equation*}
		(U(0,s), U_y(0,s), U_{yy}(0,s)) \equiv (0, -1, 0)
	\end{equation*} is a solution to a set of equations \eqref{EP2_1-p}, \eqref{EP2_2D1-p} and \eqref{EP2_2D2-p} 
satisfying the condition \eqref{constraint-p} at $s=s_0$. By uniqueness, as long as the solution $(\tau, \kappa, \xi)$ to the initial problem \eqref{modulation} with \eqref{Modul_init-p} exists, we see that  the solution $U$ continues to satisfy the constraints \eqref{constraint-p}.


Now we establish the global pointwise estimates by a bootstrap argument.

\subsection{Bootstrap argument.}\label{boots_sec}

In this section, we introduce the bootstrap assumptions for  any given  time interval $S_0:=[s_0,\sigma_1]$. 
We assume that  for  $s\in[s_0,\sigma_1]$ and $y\in\mathbb{R}$,  $(P,U)$  satisfies 
\begin{subequations}\label{Boot_2}
	\begin{align}
		&|U_y(y,s)-\overline U'(y)| \leq \frac{y^2}{10(y^2+1)},\label{EP2_1D1-p}
		\\
		&|U_y(y,s)-\overline{U}'(y)| \le \frac{1}{y^{2/3}+8}, \label{Utildey_M-p}
		\\
		& |\partial_y ^2 U(y,s)| \leq \frac{15|y|}{(y^2+1)^{1/2}}, \label{EP2_1D3-p}
		\\
		& |\partial_y ^3 U(0,s)-6| \leq 1,  \label{EP2_1D2-p} 
		\\
		& \|\partial_y ^3 U(\cdot,s)\|_{L^{\infty}} \leq M^{5/6},\label{EP2_1D5-p}
		\\
		& \|\partial_y ^4 U(\cdot,s)\|_{L^{\infty}} \leq M, \label{EP2_1D4-p}
		\\
		&	|P(y,s)| \le \frac{2A}{y^{2/3}+8},\label{Wweak-p}
	\end{align}
\end{subequations}
where $M>0$ is a sufficiently large  constant and $A>0$ is a constant defined in \eqref{Aeq}. 

The bootstrap assumption \eqref{Boot_2} implies that $U_y$ is close to $\overline{U}'$ near $y=0$ in $C^2$ norm, ensuring the constraints \eqref{constraint-p} at $y=0$, and that $U$ stays uniformly bounded in $C^4$ norm. 
In fact, the asymptotic behaviors of $U$ near $y=0$ plays an important role in resolving the degeneracy of the damping term around $y=0$ in the transport-type equations.
%

\subsection{Preliminary estimates.}
We present several simple estimates on $\overline{U}$ following from  the bootstrap assumptions \eqref{EP2_1D1-p}--\eqref{Wweak-p}.
By \eqref{EP2_1D1-p} and \eqref{y-w-y2}, we have 
\begin{equation} \label{Uy1-p}
	| U_y (y,s)| \le | \overline{U}'(y) | + | U_y(y,s) -\overline{U}'(y) | \le  \frac{1}{1+\frac{3y^2}{(1+3y^2)^{2/3}}} +\frac{y^2}{10(1+y^2)} \le 1.
\end{equation}
It follows from $U(0,s)=\overline{U}(0)=0$, $|U_y(y,s)|, |\overline{U}'(y)|\leq 1$ and \eqref{EP2_1D1-p} that
\begin{equation*}
	|U(y,s)|\leq |y|, \quad |\overline{U}(y,s)|\leq |y|, \quad |U(y,s)-\overline{U}(y)|\leq \frac{|y|^3}{30}.
\end{equation*}

It is straightforward to check that  \eqref{EP-p} with the decaying assumption $(\rho,u,\phi)(x,t)\rightarrow (1,0,0)$ as $|x|\rightarrow \infty$ has a conserved energy form 
\begin{equation*}
	H(t):=\int_{\mathbb{R}}\frac{1}{2}\rho u^2+\frac{1}{2}|\phi_x|^2+(\phi-1)e^{\phi}+1\,dx.
\end{equation*} 
That is, for any $t\geq -\veps$, we have 
\begin{equation*}
	H(t)=H(-\veps).
\end{equation*}
%

For the following Proposition, we refer to Lemma $4.2$ in \cite{BCK} and the references therein, for instance \cite{LLS}.
\begin{proposition}(\cite{BCK,LLS})\label{phi0_prop}
	For $\rho-1 \in L^{\infty}(\mathbb{R})\cap L^2(\mathbb{R})$ satisfying  $\inf_{x\in \mathbb{R}}\rho>0$ and $\lim_{|x|\rightarrow\infty}\rho=1$, let $\phi$ be the solution to the Poisson equation \eqref{EP_3-p}. Then, it holds that
	\begin{equation*}
		\int_{\mathbb{R}} |\phi_x|^2 + (\phi-1)e^{\phi} +1  \,dx \leq \frac{1}{\theta} \int_{\mathbb{R}} |\rho-1|^2\,dx,
	\end{equation*}
	where $\theta$ is given by 
	$$\theta:= \begin{cases}
		-(1-exp(\underline \rho))/\underline \rho & \text{ if } \underline \rho:=\inf_{x\in\mathbb{R}}\rho<0,
		\\
		1 & \text{ if } \underline \rho\geq 0.
	\end{cases}$$
\end{proposition} 
\begin{remark}\label{H_L2} Note that   \eqref{init_ttP-p} implies that $\rho_0$ satisfies all the necessary conditions for Proposition~\ref{phi0_prop}.  
	Then it is easy to see that  $H(t)$ is controlled by the $L^2$ norm of $(\rho_0-1,u_0)$, i.e., there exists a constant $C=C(\inf_{x\in\mathbb{R}}\rho_0, \sup_{x\in\mathbb{R}}\rho_0)>0$ such that  
	\begin{equation*}
		H(t)\leq C\|(\rho_0-1,u_0)\|_{L^2}^2 .
	\end{equation*} 
\end{remark}

Now, we derive the uniform bound of $\phi$ and $\phi_x$, for which we refer to 
Lemma~$2.1$ and Lemma~$2.2$ of \cite{BCK}. 
Let us define 
\begin{equation*}
	V(z)= \begin{cases}
		V_+(z):=\int^z_0\sqrt{2U(\tau)}\,d\tau & \text{for }z\geq 0,\\
		V_-(z):=\int^0_z\sqrt{2U(\tau)}\,d\tau & \text{for }z\leq 0,
	\end{cases}
\end{equation*}
where $U(\tau):=(\tau-1)e^{\tau}+1\geq 0$.
\begin{lemma}\label{phi_lem}
	As long as the solution to \eqref{EP-p} exists, we have
	\begin{equation}
		\| \phi (\cdot, t) \|_{L^\infty}, \|\Phi(\cdot,s)\|_{L^{\infty}}\leq M_1, \label{Phi_0_M}
	\end{equation} 
	and   if additionally \eqref{Wweak-p} holds, then
	\begin{equation}
		\|\Phi_y(\cdot,s)\|_{L^{\infty}}\leq M_2e^{-11s/8} 
		\label{Phi_1_M-p}
	\end{equation}
	 for some positive constants $M_1$ and $M_2$ depending on $\|(\rho_0-1,u_0)\|_{L^2}$, $\inf_{x\in\mathbb{R}}\rho_0$ and $\sup_{x\in\mathbb{R}}\rho_0$.
\end{lemma}

\begin{proof}
Loosely following the proof of Lemma $2.1$ in \cite{BCK}, we see that $V$ satisfies
	\begin{equation*}
		\begin{split}
			0 \leq V(\phi(x,t)) & \leq \int^x_{-\infty}\left| V'(\phi(y,t))\right|\left|\phi_{x} (y,t) \right|\,dy 
			\\
			&\leq \int^\infty_{-\infty}U(\phi(x,t))\,dx+\frac{1}{2}\int^\infty_{-\infty}\left|\phi_x (x,t) \right|^2\,dx\\
			&\leq \int^{\infty}_{-\infty} \frac{1}{2}\rho u^2+\frac{1}{2}|\phi_x|^2+(\phi-1)e^\phi+1\,dx=H(t)=H(-\veps).
		\end{split}
	\end{equation*} 
	Thus, we see that  $\|\phi (\cdot, t) \|_{L^{\infty}}\leq \max\{|V^{-1}_{-}(H(-\veps))|, |V^{-1}_{+}(H(-\veps))|\}$, which  in view of Remark~\ref{H_L2}, implies \eqref{Phi_0_M} for some $M_1>0$ which depends only on $\|(\rho_0-1,u_0)\|_{L^2}$, $\inf_{x\in\mathbb{R}}\rho_0$ and $\sup_{x\in\mathbb{R}}\rho_0$.
	
	Multiplying \eqref{EP_3-p} by $-\phi_x,$ and then integrating in $x$,
	\begin{equation}\label{phix2}
		\begin{split}
			\frac{\phi_x^2}{2}&=\int^{x}_{-\infty}-\phi_x(\rho-1)\,dx'+\int^x_{-\infty}e^{\phi}\phi_x-\phi_x \,dx'\\
			&\leq \left(\int^{x}_{-\infty}\phi_x^2\, dx'\right)^{1/2}\left(\int^{x}_{-\infty}(\rho-1)^2dx'\right)^{1/2}+e^{\phi}-\phi-1\\
			&\leq \sqrt{2 H(-\ve) }\left(\int^{x}_{-\infty}(\rho-1)^2 \,dx'\right)^{1/2}+e^{M_1}-M_1-1.
		\end{split}
	\end{equation}
	Here we have used the fact 
	that 
	\begin{equation*}
		\frac{1}{2}\int_{\mathbb{R}}|\phi_x|^2dx'\leq H(-\veps)
	\end{equation*} and the fact from \eqref{Phi_0_M} that 
	\begin{equation*}
		e^{\phi}-\phi-1 \leq e^{M_1}-M_1-1. 
	\end{equation*}
	On the other hand, we see from \eqref{WZPhi-p} and \eqref{Wweak-p} that 
	\begin{equation*}
		\int^x_{-\infty}(\rho-1)^2dx'=e^{-3s/2}\int^y_{-\infty}(\rho-1)^2dy' \leq (2A)^2e^{s/2}\int^{\infty}_{-\infty}\frac{1}{(y^{2/3}+8)^2}\,dy=Ce^{s/2}
	\end{equation*}
	for some $C>0$. 
	%
	This, together with \eqref{phix2}, implies  that 
		$|\phi_x| 
		\le M_2 e^{s/8},$ 
	 and in turn,  
	\begin{equation*}
		|\Phi_y|=|\phi_x e^{-3s/2}|\leq M_2e^{-11s/8}
	\end{equation*} for some positive constant $M_2$.
	This completes the proof. 
\end{proof}

 Using Lemma~\ref{phi_lem}, we obtain the decay estimate of $\dot{\tau}$. 
	\begin{lemma}
		 Suppose that \eqref{Wweak-p} holds. Then 
		it holds that
		\begin{equation}\label{tau_dec}
			|\dot{\tau}|\leq Ce^{-s}
		\end{equation}
		and
		\begin{equation}\label{d-t-e-p}
			\frac{1}{1-\dot{\tau}}\leq  1+O(\veps).
		\end{equation}
	\end{lemma}
\begin{proof}
	Using \eqref{EP2_3-p} and \eqref{Eq_Modul1-p} with \eqref{Wweak-p} and \eqref{Phi_0_M}, we have
	\begin{equation*}
		\begin{split}
			|\dot{\tau}|
			= |e^s\Phi_{yy}(0,s)| &\leq |P(0,s)|e^{-s} + e^{\|\Phi\|_{L^{\infty}}}e^{-2s} + e^{-2s}
			\leq Ce^{-s}. 
		\end{split}
	\end{equation*}
	Then, \eqref{d-t-e-p} immediately follows.
\end{proof}

\subsection{Global continuation}\label{global-conti}
Now we shall close  the bootstrap assumptions \eqref{EP2_1D1-p}--\eqref{Wweak-p}.  
\begin{proposition}\label{mainprop}
	There is a sufficiently small constant $\veps_0=\veps_0(M, \|(\rho_0-1,u_0)\|_{L^2}, \inf_{x\in\mathbb{R}}\rho_0, \sup_{x\in\mathbb{R}}\rho_0, \|\rho_0\|_{C^3})>0$ such that for each $\veps\in (0,\veps_0)$, the following statements hold. Consider the smooth solution $(\rho,u,\phi)$ to \eqref{EP-p} satisfying \eqref{init_w_3-p}--\eqref{init_ttP-p}. Suppose that the associated solution $(P,U)$ in the self-similar variables satisfies  \eqref{EP2_1D1-p}--\eqref{Wweak-p}  for  $s\in[s_0,\sigma_1]$ and $y\in\mathbb{R}$. Then,  for  $s\in[s_0,\sigma_1]$ and $y\in\mathbb{R}$, we have  
	\begin{subequations}
		\begin{align}
			&|U_y(y,s)-\overline U'(y)| \leq \frac{y^2}{20(y^2+1)},\label{1}
			\\
			&|U_y(y,s)-\overline{U}'(y)| \le \frac{24}{25}\frac{1}{y^{2/3}+8}, \label{Utildey_M-str}
			\\
			& |\partial_y ^2 U(y,s)| \leq \frac{14|y|}{(y^2+1)^{1/2}}, \label{3}
			\\
			& |\partial_y ^3 U(0,s)-6| \leq C\veps,  \label{4} 
			\\
			& \|\partial_y ^3 U(\cdot,s)\|_{L^{\infty}} \leq \frac{M^{5/6}}{2},\label{5}
			\\
			& \|\partial_y ^4 U(\cdot,s)\|_{L^{\infty}} \leq \frac{M}{2}, \label{6}
			\\
			&|P(y,s)| \le \frac{ A}{y^{2/3}+8}.\label{Wweak-str}
		\end{align}
	\end{subequations}
	\end{proposition}
\begin{proof}
	Since the proof is rather lengthy,  we decompose it into several lemmas given in Section~\ref{sec3-boot}. The desired estimates are established through Lemma~\ref{U30lem}--Lemma~\ref{mainprop_1-p}.
	For the sake of readers, we give a list: \eqref{1} is obtained in Lemma~\ref{Wy1_lem_p}, \eqref{Utildey_M-str} in Lemma~\ref{mainprop_1-p}, \eqref{3} in Lemma~\ref{Wy2_lem}, \eqref{4} in Lemma~\ref{U30lem}, 
	\eqref{5} in Lemma~\ref{w-3-est}, 
		\eqref{6} in Lemma~\ref{Wy4_lem-p} 
		and \eqref{Wweak-str} in Lemma~\ref{propW}, respectively. 
%
%
%
\end{proof}

Next, we present a local existence theorem for the initial value problem \eqref{EP-p} with the initial data \eqref{in-H-C}. 
\begin{lemma}\label{criterion} (\cite{LLS}). 
	For the initial data \eqref{in-H-C}, there is $T\in (-\veps,\infty)$ such that the initial value problem \eqref{EP-p} admits a unique solution 
	$(\rho-1, u,\phi )\in C([-\veps,T); H^4(\mathbb{R})\times H^5(\mathbb{R}) \times H^6(\mathbb{R} ))$.
	Suppose further that
	\begin{equation*}
		\lim_{t\nearrow T}\|u_x(\cdot, t)\|_{L^{\infty}}<\infty.
	\end{equation*}
	Then, the solution can be continuously extended beyond $t=T$, i.e., there exists a unique solution
	\begin{equation*}
		(\rho-1, u, \phi)\in C([-\veps,T+\delta); H^4(\mathbb{R})\times H^5(\mathbb{R}) \times H^6(\mathbb{R} ))
	\end{equation*}
	for some $\delta>0$.
\end{lemma}
\begin{proof}
	We refer to Theorem~1 of \cite{LLS} for the local existence. 
	By revisiting the energy estimates, for any $k\ge2$, 
	we obtain 
	\begin{equation*}
		\frac{d}{dt}\left(\|\rho-1\|^2_{H^k}+\|u\|^2_{H^{k+1}}\right)\leq C(\|\rho\|_{L^{\infty}},  \|u_x\|_{L^{\infty}}, \|\phi\|_{L^{\infty}}) \left(\|\rho-1\|^2_{H^k}+\|u\|^2_{H^{k+1}}\right), 
	\end{equation*}
	where $C=C(\|\rho\|_{L^{\infty}},  \|u_x\|_{L^{\infty}}, \|\phi\|_{L^{\infty}})>0$ is a constant depending on $\|\rho\|_{L^{\infty}}, \|u_x\|_{L^{\infty}}$ and $\|\phi\|_{L^{\infty}}$. 
On the other hand,  we see that
	 \eqref{Phi_0_M} in Lemma~\ref{phi_lem} implies $\|\phi\|_{L^{\infty}}\lesssim 1$ and the continuity equation \eqref{EP_1-p} gives $\| \rho(\cdot, t)\|_{L^\infty} \leq e^{ \int_0^t \| u_x (\cdot, t' )\|_{L^\infty} \,dt'} \| \rho_0\|_{L^\infty}$. Thus, as long as $\| u_x(\cdot, t) \|_{L^\infty}<\infty$, we have $\| \rho(\cdot, t)\|_{L^\infty} <\infty$, and $C=C(\| u_x \|_{L^\infty})<\infty$, which yields the desired  extension criteria.
\end{proof}
Note that by the Sobolev embedding, 
if $(\rho-1,u)\in C([-\veps,T); H^4(\mathbb{R}) \times H^5(\mathbb{R}) )$, then 
$\rho-1 \in C^3([-\veps,T)\times \mathbb{R})$ and $u\in C^4([-\veps,T)\times \mathbb{R})$. 
 By Lemma~\ref{criterion}, we see that there is a time-interval $[s_0,\sigma_1]$ for some $\sigma_1>s_0$, in which the solution $(P,U)$ satisfies the bootstrap assumptions \eqref{EP2_1D1-p}--\eqref{Wweak-p}.

Now, we show that the desired estimates in Proposition~\ref{mainprop} holds  for all $s\ge s_0$. 
Let us define  the vector $\{V_i(s)\}_{1\leq i\leq 7}$ as the terms on the left hand side of \eqref{EP2_1D1-p}--\eqref{Wweak-p} with their corresponding weight functions, i.e.,
\begin{equation*}
	\begin{split}
		V_1 &:= \sup_{y\in\mathbb{R}} \left( \frac{1+y^2}{y^2} |U_y(y,s)-\overline{U}'(y)| \right) , \quad V_2 := \sup_{y\in\mathbb{R}} \left( (y^{2/3} + 8 ) | U_y(y,s) - \overline U'(y) | \right),
		\\
		V_3 & : = \sup_{y\in\mathbb{R}} \left( \frac{(1+y^2)^{1/2}}{|y|} | U_{yy}(y,s) | \right), \quad V_4  := | \partial_y^3 U(0,s) -6 |,  \quad V_5  := \| \partial_y^3 U (\cdot, s) \|_{L^\infty}, 
		\\
		V_6 & := \| \partial_y^4 U (\cdot, s) \|_{L^\infty},   \quad 
		V_7  :=  \sup_{y\in\mathbb{R}} \left( (y^{2/3}+8)|P(y,s)| \right).
	\end{split}
\end{equation*}
Also, 
let $\{K_i\}_{1\leq i\leq 7}$ be the right hand side terms of \eqref{EP2_1D1-p}--\eqref{Wweak-p}, which are 
\begin{equation*}
	K_1:= \frac{1}{10}, \ \ K_2:= 1,\ \ K_3:= 15, \ \ K_4:=1 , \ \  K_5:= M^{5/6}, \ \ K_6 := M,  \ \ K_7:= 2A. 
\end{equation*}
For $\beta=(\beta_1,\cdots , \beta_7)\in (0,1)\times \cdots \times (0,1)$, we define the solution space $\mathfrak{X}_\beta(s)$ by
\begin{multline*}
\mathfrak{X}_\beta(s) :=\{ (P,U) \in (C^3\times C^4)([s_0,s]\times \mathbb{R}) : V_i(s') < \beta_i K_i, \;\forall s'\in[s_0,s], \; 1\leq i\leq 7  \}.
\end{multline*}
For simplicity, let us denote $\mathfrak{X}_\beta(s)$ with $\beta_i=1$ for $i=1,\cdots,7$ by $\mathfrak{X}_1(s)$.

Note that the set of initial conditions \eqref{init_w_3-p}--\eqref{init_ttP-p} for \eqref{EP-p} implies that the new unknowns $(P,U, \Phi)$  in  the self-similar variables $(y,s)$  satisfy  \eqref{EP2_1D1-p}--\eqref{Wweak-p} at $s=s_0 = -\log \veps$. 
That is,  
\begin{equation*}
	V_i(s_0)\leq \alpha_iK_i 
\end{equation*}
 for some $\alpha_i\in(0,1)$, $1\leq i\leq 7$.  
Proposition~\ref{mainprop} indicates that if $(P,U)\in \mathfrak{X}_1(\sigma_1)\cap \mathfrak{X}_\alpha(s_0)$, then $(P,U)\in \mathfrak{X}_\beta(\sigma_1)$ for some $\beta_i\in (\alpha_i,1)$, $1\leq i\leq 7$. 
 
 We will use a standard continuation argument together with the local existence theory, Lemma~\ref{criterion} and the initial conditions \eqref{in-H-C}--\eqref{init_ttP-p}, to show that the solution $(P,U)$ exists globally and satisfies Proposition~\ref{mainprop} for all $s\in[s_0,\infty)$. I.e., $(P,U) \in \mathfrak{X}_\beta(\infty)$.

 Now, we define
 \begin{equation*}
 	\sigma:= \sup \{ s\ge s_0 : 
 	(P,U) \in \mathfrak{X}_\beta(s) \}.
 \end{equation*} 
 We claim that $\sigma=\infty$. 
 Suppose to the contrary that $\sigma<\infty$. Note that for the corresponding $t=T_\sigma$ of $s=\sigma$, it is easy to see that thanks to \eqref{change_var}, $\sigma<\infty$ implies $\tau(t)>t$  for all $t\in[-\veps,T_\sigma]$. This means that $(y,s)\in \mathbb{R}\times [s_0,\sigma]$ is well-defined in $(x,t)\in \mathbb{R}\times[-\veps,T_\sigma]$. 
 Using \eqref{EP2_1D1-p} and \eqref{y-w-y2}, we have 
 \[ \sup_{s\in[s_0, \sigma)} \| U_y(\cdot,s) \|_{L^\infty} \le \sup_{s\in[s_0, \sigma)} \| U_y(\cdot,s)-\overline{U}'(\cdot) \|_{L^\infty} +\| \overline{U}'(\cdot) \|_{L^\infty}\le \frac{y^2}{10(1+y^2)}+ \frac{1}{1+\frac{3y^2}{(3y^2+1)^{2/3}}}\le 1. \]
 This together with the fact that $u_x = e^{s} U_y$ implies $
 \|u_x(\cdot, t)\|_{L^\infty}\le e^{\sigma}<\infty$ for all $t\le T_\sigma$.
 %
 Then, thanks to Lemma~\ref{criterion}, 
 there is a unique extended solution $(\rho-1,u)\in C^3\times C^4([-\ve, T_\sigma+\delta))
 $ for some $\delta>0$.
This immediately implies that 
 $(P,U)$ can be extended so that $(P,U)\in \mathfrak{X}_1(\sigma+\delta')$ for some $\delta'>0$. Then, by Proposition~\ref{mainprop},
 we see that $(P,U)\in \mathfrak{X}_\beta(\sigma+\delta')$. 
  This contradicts to the definition of $\sigma$, which concludes that  $\sigma=\infty$. 

Now we are ready to  prove Theorem~\ref{mainthm1}.

	\subsection{Proof of Theorem~\ref{mainthm1}}\label{thm-pf}
We split the proof into several subsections. 

Step 1: 
Recall the definition of $T_\ast$ in \eqref{T-star}, i.e., $T_\ast:=  \inf\{ t\ge-\ve : \tau(t)=t\}$. 
For sufficiently small $\ve>0$, thanks to  \eqref{tau_dec} implying $|\dot{\tau}(t)|<1$,  and $\tau(-\veps)=0$, there exists a number $t_0<\infty$ such that $\tau(t_0)=t_0$. This means that  $T_\ast<\infty$ is well-defined. 
By a standard continuation argument in Section~\ref{global-conti}, we have shown that  $(P,U)\in\mathfrak{X}_\beta(\infty)$. 
%
This implies that the solution $(\rho-1,u)$ exists at least up to $t=T_\ast$, and that   
\begin{equation*}
	(\rho-1,u)\in  C^3([-\veps,T_*)\times \mathbb{R})\times C^4([-\veps,T_*)\times \mathbb{R}). 
\end{equation*}
(In fact, in Step 2, we prove that $T_\ast$ is the maximal existence time of the smooth solution $(\rho,u)$ by showing $\| u_x(\cdot, t) \|_{L^\infty} \nearrow \infty$ as $t\nearrow T_*$.)
   
%
%
%

Furthermore, we claim that $|T_\ast| = O( \ve^{2})$.
	By \eqref{tau_dec} together with  existence of the solution up to $T_\ast$, we see that  
\begin{equation*}
|\dot{\tau}|\leq Ce^{-s}, \quad t\in [-\ve, T_\ast]. 
\end{equation*}	
	  Using  this, we have 
	\[
	|\tau(t)|\leq \int^{t}_{-\veps}|\dot{\tau}(t')|\,dt' \leq C\veps \int^{s}_{-\log\veps} e^{-s'}\,ds' \leq  C\veps^{2},
	\] where we have used \eqref{dx-dy}. 
	Since  $\tau(T_\ast)=T_\ast$, we see that $|T_\ast|=O(\veps^{2})$. 
	Moreover, thanks to Lemma~\ref{UW_far-p}, we see that $\xi(t)$ and $\kappa(t)$ converge respectively to $\xi(T_\ast)\in\mathbb{R}$ and $\kappa(T_\ast)\in\mathbb{R}$ as $t\nearrow T_\ast$. 
	%
	%
	%
	%
	%
	%

Step 2: We prove the assertion $(i)$ and that $T_\ast$ is the maximal existence time.
We first show that $u$ has $C^{1/3}$ regularity at the blow-up time. By a standard continuation argument as in Section~\ref{global-conti} and Proposition~\ref{mainprop}, we see that \eqref{Utildey_M-str} holds globally, i.e.,  
\begin{equation*}
	(y^{2/3}+8)|U_y(y,s)-\overline{U}'(y)| \le \frac{24}{25} \quad  \text{ for all } s\ge s_0, \ \  y\in \mathbb{R}.
\end{equation*}
From this and \eqref{y-w-y}, we have 
\begin{equation}\label{Wydec_fin_2}
	|U_y (y,s) | < \frac{1}{y^{2/3}+8}+|\overline{U}'(y) |\leq \frac{C}{y^{2/3}+1} 
\end{equation} 
for some   constant $C>0$. 
Using this, we have  
\begin{equation}\label{not-zero}
	{\frac{|U(y,s)-U(y',s)|}{{|y-y'|}^{1/3}}} =  \frac{1}{{|y-y'|}^{1/3}} {\left|\int_{y'}^y U_y(\tilde y,s) d\tilde y \right|} \le   \frac{C}{{|y-y'|}^{1/3}}  \left| \int_{y'}^{y}{(1+\tilde y^2)^{-1/3}\,d\tilde y} \right| \lesssim 1.
\end{equation} 
Consider any two points $x\neq x'\in \mathbb{R}$. 
By the change of variables as \eqref{change_var} and \eqref{WZPhi-p}, we see that  \eqref{not-zero} implies 
\begin{equation*}\label{Eq_holder}
	\frac{|u(x,t)-u(x',t)|}{{|x-x'|}^{1/3}}=\frac{|U(y,s)-U(y',s)|}{{|y-y'|}^{1/3}} \lesssim 1
\end{equation*}
for all $t\in[-\ve, T_\ast]$. 
Hence, we get   
\begin{equation*}
	\sup_{t\in[-\varepsilon, T_\ast]} [u(\cdot, t)]_{C^{1/3}(\mathbb{R}) } \lesssim 1, 
\end{equation*}
i.e., $u \in L^{\infty}([-\varepsilon,T_\ast];C^{1/3}(\mathbb{R}))$.


Next, we claim that when $\beta > 1/3$, $C^\beta$ H\"older norm of $u$ blows up as $t\nearrow T_\ast$. 
We note from  \eqref{change_var} and \eqref{WZPhi-p} that   
\begin{equation*}
	\frac{|u(x,t)-u(x',t)|}{|x-x'|^\beta}=   e^{(\frac{3}{2}\beta-\frac{1}{2})s} \frac{|U(y,s)-U(y',s)|}{|y-y'|^{\beta}}. 
\end{equation*} 
%
 Let $y'=0$ and $y \in(0,\delta_0)$ for some $\delta_0\ll1$. 
By the mean value theorem,
\begin{equation*}
	\frac{|U(y,s)-U(0,s)|}{|y|^{\beta}}=|U_y(\overline{y},s)||y|^{1-\beta} 
\end{equation*}
for some $\overline{y}\in (0,\delta_0)$.
From \eqref{constraint-p} and \eqref{EP2_1D1-p},
we see that $|U_y(\overline{y},s)| \geq 1/2$ for $\delta_0\ll1$, which implies that $|U_y(\overline{y},s)||y|^{1-\beta}\ge c_0$ for some $c_0>0$. 
In view of this, we have 
\begin{equation}\label{LB-blow}
	[u(\cdot, t)]_{C^{\beta}(\mathbb{R})}\geq  e^{(\frac{3}{2}\beta-\frac{1}{2})s}  \frac{|U(y,s)-U(0,s)|}{|y|^{\beta}}= e^{(\frac{3}{2}\beta-\frac{1}{2})s} |U_y(\overline{y},s)||y|^{1-\beta}  \ge  c_0  e^{(\frac{3}{2}\beta-\frac{1}{2})s}. 
\end{equation}
Since $\beta>1/3$, 	we conclude that $[u(\cdot, t)]_{C^{\beta}(\mathbb{R})}\to\infty$ as $t\nearrow T_\ast$, i.e., $s\to \infty$. 

Step 3: We shall prove the assertion $(ii)$. 
We note that  
\begin{equation*} 
	|U(y,s)-U(y',s)| \le \int_{y'}^y |U_y (\tilde y,s) | d\tilde y \le C \int_{y'}^y  (1+\tilde y^2)^{(\beta-1)/2} d\tilde y \leq C |y-y'|^\beta,
\end{equation*} 
where the second to the last inquality holds thanks to  \eqref{Wydec_fin_2}, i.e., $$|U_y (y,s) |  \leq C(1+y^2)^{-1/3} \le C (1+y^2)^{(\beta-1)/2}$$ for $\beta>1/3$. 
Then again from  \eqref{change_var} and \eqref{WZPhi-p}, we have
\begin{equation} \label{UB-blow}
	\begin{split}
		\frac{|u(x,t)-u(x',t)|}{|x-x'|^\beta} & = e^{(\frac{3}{2}\beta-\frac{1}{2})s} \frac{|U(y,s)-U(y',s)|}{|y-y'|^{\beta}} \le C e^{(\frac{3}{2}\beta-\frac{1}{2})s}
	\end{split} 
\end{equation} 
for some $C>0$. 
Note  from \eqref{tau_dec} that 
\begin{equation*}\label{tau-T*}
	\frac{1}{2}(T_\ast-t) \leq \tau(t) - t = \int_t^{T_\ast} (1-\dot{\tau})\,dt' \leq \frac{3}{2}(T_\ast-t).
\end{equation*}
This, together with \eqref{LB-blow} and \eqref{UB-blow}, yields the desired blow-up rate as 
\[\left[ u(\cdot, t) \right]_{C^\beta} \sim (T_\ast-t)^{-\frac{3\beta-1}{2}}.\]
 Thanks to \eqref{xibound} in Lemma~\ref{UW_far-p}, $\xi(t)$ converges to $\xi(T_\ast)$ as $t\nearrow T_\ast$ and $x=\xi(T_\ast)$ is the blow up location of $u_x$. In particular, the blow up rate of $\| \partial_x u(\cdot, t) \|_{L^\infty}$ is $(t-T_\ast)^{-1}$.

Step 4: We shall prove $(iii)$. 
We first note by \eqref{Wweak-str} in Lemma~\ref{mainprop} that 
	\begin{equation*}
		\left((x-\xi(t))^{2/3}  +  (\tau(t)-t)\right)|\rho(x,t)-1|=(y^{2/3}+1)|P(y,s)|\lesssim 1. 
	\end{equation*}  
	On the other hand, we show  that $(y^{2/3}+1)|P(y,s)|$ has a uniform positive lower bound near $y=0$, i.e.,  there exists a constant $\Lambda=\Lambda(\overline{s},h)>0$ satisfying 
	\begin{equation}\label{claim'}
				(y^{2/3}+1)|P(y,s)| \ge   \Lambda, \qquad s\geq \overline{s}, \quad |y|\leq h
	\end{equation}
 for some $\overline{s}>s_0,h>0$. 
	This will complete the proof of $(iii)$.

	In what follows, we show \eqref{claim'}. 
%
We use the following three inequalities from  Lemma ~\ref{Pylem}, ~\ref{Phider_lem} and ~\ref{UW_far-p} in Section~\ref{sec3-boot}:
\begin{equation}\label{C123}
\begin{split}	& \|P_y(\cdot,s)\|_{L^{\infty}}+\|P_{yy}(\cdot,s)\|_{L^{\infty}}\leq C, 
\\
& \|\Phi_{yy}(\cdot,s)\|_{L^{\infty}} \leq C e^{-2s}, \quad |e^{s/2}(\kappa-\dot{\xi})|\leq C e^{-s},
\end{split}
\end{equation}
where $C>0$ is some   constant.
%
%
%
	From \eqref{EP2_1-p}, we have  
	\begin{equation}\label{rhoeq}
		\partial_s(e^sP+1)+\left(\frac{U}{1-\dot{\tau}}+\frac{3}{2}y\right)(e^sP+1)_y+\frac{U_y}{1-\dot{\tau}}(e^sP+1)=-\frac{e^{s/2}(\kappa-\dot{\xi})}{1-\dot{\tau}}e^sP_y.
	\end{equation}
	Let $\psi(y;s)$ be the  characteristic curve satisfying 
\begin{equation}\label{psi-00}
 \frac{d}{ds} \psi (y;s)=\frac{U(\psi(y;s),s)}{1-\dot{\tau}}+\frac{3}{2}\psi(y;s), \qquad  \psi(y;s_0)=y.
 \end{equation}  
	By integrating \eqref{rhoeq} along $\psi$, we have 
	\begin{multline*}
		(e^sP+1)(\psi(y;s),s)
  \\
  =(e^{s_0}P(y,s_0)+1)e^{-\int^s_{s_0}\frac{U_y(\psi(y;s'),s')}{1-\dot{\tau}}\,ds'} 
  -\int^s_{s_0}e^{-\int^s_{s'}\frac{U_y(\psi(y;s''),s'')}{1-\dot{\tau}}\,ds''}\left(\frac{e^{s'/2}(\kappa-\dot{\xi})}{1-\dot{\tau}}e^{s'}P_y(\psi(y;s'),s')\right)\,ds'.
	\end{multline*}
	 We note that when $y=0$, the trivial function $\psi(0;s)\equiv0$ is a solution to  \eqref{psi-00}. 
	 Using this together with $U_y(0,s)=-1$, we have 	 
	\begin{equation}\label{rho_integ}
		\begin{split}
			(e^sP+1)(0,s)&=(e^{s_0}P(0,s_0)+1)e^{\int^s_{s_0}\frac{1}{1-\dot{\tau}}\,ds'}\underbrace{-\int^s_{s_0}e^{\int^s_{s'}\frac{1}{1-\dot{\tau}}\,ds''}\left(\frac{e^{s'/2}(\kappa-\dot{\xi})}{1-\dot{\tau}}e^{s'}P_y(0,s')\right)\,ds'}_{=:I_1}.
		\end{split}
	\end{equation}
	Now, to evaluate  $I_1$, we first derive the $s$-decay of $P_y(0,s)$. By substituting $y=0$ into \eqref{EP2_3D1-p} with \eqref{constraint-p}, we find  
	\begin{equation*}
		\partial_sP_y(0,s)+\left(\frac{5}{2}-\frac{2}{1-\dot{\tau}}\right)P_y(0,s)=-\frac{e^{s/2}(\kappa-\dot{\xi})}{1-\dot{\tau}}P_{yy}(0,s),
	\end{equation*}
	  which we integrate  to have 
	\begin{equation}\label{Py0}
		\begin{split}
			|P_y(0,s)|&\leq |P_y(0,s_0)|e^{-\int^s_{s_0}\left(\frac{5}{2}-\frac{2}{1-\dot{\tau}}\right)\,ds'}+\int^s_{s_0}e^{-\int^s_{s'}\left(\frac{5}{2}-\frac{2}{1-\dot{\tau}}\right)\,ds''}\frac{e^{s'/2}|\kappa-\dot{\xi}|}{1-\dot{\tau}}\left|P_{yy}(0,s')\right|\,ds'
			\\
			&\leq |P_y(0,s_0)|e^{-(s-s_0)/4}+C\int^s_{s_0}e^{-(s-s')/4}e^{-s'}\,ds'
			\\
			&\leq \veps^{-1/4}|P_y(0,s_0)|e^{-s/4}+C\veps^{3/4} e^{-s/4}
			\\
			&\leq  {(C\veps^{3/4}+\veps^{-1/4}|P_y(0,s_0)| )}  e^{-s/4}.
		\end{split}
	\end{equation}
	Here, we used \eqref{tau_dec} and \eqref{C123} for the second inequality.
	Then using \eqref{tau_dec}, \eqref{C123} and \eqref{Py0}, we obtain
	\begin{equation*}
		\begin{split}
			|I_1|&\leq e^{\int^s_{s_0}\frac{1}{1-\dot{\tau}}\,ds'}\int^{s}_{s_0}e^{s'}\left|P_y(0,s')\right|\frac{e^{s'/2}|\kappa-\dot{\xi}|}{1-\dot{\tau}}\,ds' 
			\\
			&\leq C e^{\int^s_{s_0}\frac{1}{1-\dot{\tau}}\,ds'}\int^{s}_{s_0}\left|P_y(0,s')\right|\,ds'
			\\ 
			&\leq  C\veps\left(1+\veps^{-1}|P_y(0,s_0)|\right) e^{\int^s_{s_0}\frac{1}{1-\dot{\tau}}}\,ds'.
		\end{split}
	\end{equation*}
 Using this for  \eqref{rho_integ}, one has 
	\begin{equation}\label{P-low}
		\begin{split}
				e^sP(0,s)+1&\geq \left(\rho_0(0)- C\veps (1+\veps^{3/2}| \partial_x \rho_0 (0)| )\right)e^{\int^s_{s_0}\frac{1}{1-\dot{\tau}}\,ds'} 
				 \geq  \frac{ \rho_0(0) \veps }{4}e^s .
		\end{split}
	\end{equation}
	Here, we used  \eqref{tau_dec}  to get 
		\begin{equation*}
		 e^{\int^s_{s_0}\frac{1}{1-\dot{\tau}}\,ds'} \geq  e^{\int_{s_0}^s \frac{1}{1+Ce^{-s'}} \,ds' } \geq  e^{\int_{s_0}^s 1-Ce^{-s'}\, ds' } =\veps e^s e^{C(e^{-s}-\veps)} \geq \veps e^{-C\veps}e^s\geq  \frac{\ve}{2} e^{s}
	\end{equation*}
	for sufficiently small $\veps>0$.
	\eqref{P-low} implies that there is $\Lambda_0=\Lambda_0(\ve)>0$ such that $P(0,s)\ge \Lambda_0$ for any sufficiently large $s$. Thanks to \eqref{C123}, i.e., $P_y$ is uniformly bounded,   there exists $(\overline{s},h)$ such that  \eqref{claim'} holds. We are done. 
\qed

\section{Closure of Bootstrap assumptions}\label{sec3-boot}
In this section, we close the bootstrap assumptions. First we begin with some crude estimates for $P$ and its derivatives. 
\begin{lemma}\label{Lemma1}
Under the assumptions of Proposition~\ref{mainprop}, it holds that for all $s\in[s_0,\sigma_1]$,   
	\begin{equation}\label{EP2_1P2}
		\|P(\cdot,s)\|_{L^{\infty}} \leq (1+C\veps)\|P(\cdot,s_0)\|_{L^{\infty}}+C\veps.
	\end{equation}
\end{lemma}

\begin{proof}

Rewriting \eqref{EP2_1-p}, we obtain
\begin{equation*}
	\partial_s P+ D^{P} P +\mathcal{U} P_y= F^{P}
\end{equation*}
with $\mathcal{U}$ defined in \eqref{U-p}, i.e., 
\begin{equation*}\label{U-p-1}
	\mathcal{U}:= \frac{U}{1-\dot{\tau}} + \frac{3}{2}y + \frac{e^{s/2}(\kappa-\dot{\xi})}{1-\dot{\tau}} 
\end{equation*}
and
\[
D^{P} := 1 + U_y, \quad F^{P} :=\frac{-\dot{\tau}U_yP}{1-\dot{\tau}}- \frac{e^{-s}U_y}{1-\dot{\tau}}.
\]
Integrating the above equation along $\psi(y;s)$ satisfying $\frac{d}{ds}\psi=\mathcal{U}(\psi,s)$ and $\psi(y;s_0)=y$, we obtain
\begin{equation}\label{Ptemp'}
	P(\psi(y;s),s)=P(y,s_0)e^{-\int^s_{s_0}(D^{P}\circ \psi)\,dr}+\int^s_{s_0}e^{-\int^s_{s'}(D^{P}\circ\psi)\,dr}(F^{P}\circ\psi)\,ds'.
\end{equation}
We note from \eqref{Uy1-p}  that 
$D^{P} \geq 0$, 
 and 
 from  \eqref{Uy1-p}, \eqref{tau_dec}, \eqref{d-t-e-p} 
  that 
\begin{equation*}
	\|F^{P}\|_{L^{\infty}}\leq \frac{1}{1-\dot{\tau}}|\dot{\tau}|\|U_y\|_{L^{\infty}}\|P\|_{L^{\infty}} +\frac{1}{1-\dot{\tau}}e^{-s}\|U_y\|_{L^{\infty}}\leq C(1+ \|P\|_{L^{\infty}}) e^{-s}. 
\end{equation*}
Combining all for \eqref{Ptemp'} yields the desired estimate \eqref{EP2_1P2}.  
\end{proof}

\begin{lemma}\label{Pylem}
	Under the assumptions of Proposition~\ref{mainprop}, we have 
	\begin{equation}
		\|\partial_y^jP(\cdot,s)\|_{L^{\infty}} \leq C_j \qquad (j=1,2,3) \label{EP2_Pder}
	\end{equation}
	for all $s\in[s_0,\sigma_1]$. Here,  $C_j=C_j(M, \sum ^{j}_{i=0}\|\partial_y^iP(\cdot,s_0)\|_{L^{\infty}})$ are positive constants. 
\end{lemma}
\begin{proof}
	Recalling \eqref{EP2_3D1-p}--\eqref{EP2_3D3-p}, $\partial_y^j P$ satisfies 
	\begin{equation}\label{Pder_eq}
		\partial_s \partial_y^j P + D^P_j \partial_y^j P +\mathcal{U}\partial_y^{j+1} P = F^P_j, 
	\end{equation}
	where $\mathcal{U}$ is defined in \eqref{U-p}, and  $D^P_j$ and $F^P_j$, $j=1,2,3$, are given by 
		\begin{subequations}
		\begin{align*}
			&D^P_1(y,s):=\frac{5}{2}+\frac{2U_y}{1-\dot{\tau}},
			\quad D^P_2(y,s):=4+\frac{3U_y}{1-\dot{\tau}},
			\quad D^P_3(y,s):=\frac{11}{2}+\frac{4U_y}{1-\dot{\tau}}
		\end{align*}
	\end{subequations}
	and 
	\begin{subequations}
		\begin{align*}
			&F_1^P (y,s):= - \frac{U_{yy}P}{1-\dot{\tau}} - \frac{e^{-s}U_{yy}}{1-\dot{\tau}},
			\\
			&F^P_2(y,s):= - \frac{ 3U_{yy}P_{y}+\partial_y^3UP}{1-\dot{\tau}}-\frac{e^{- s}}{1-\dot{\tau}}\partial_y^3U,
			\\
			&F^P_3(y,s):= - \frac{\partial_y^4UP+4\partial_y^3UP_y + 6U_{yy}P_{yy}}{1-\dot{\tau}}-\frac{e^{- s}}{1-\dot{\tau}}\partial_y^4U. 
		\end{align*}
	\end{subequations}
	We note that by \eqref{Uy1-p} and \eqref{d-t-e-p},  the lower bounds for $D_j^P$ are obtained as 
	\begin{subequations}
		\begin{align*}
			&D^P_1(y,s)\geq 	\frac{5}{2}-2(1+O(\veps))\geq \frac{1}{4} =: \lambda_1,
			\\
			&D^P_2(y,s) \geq 4-3(1+O(\veps))\geq \frac{1}{2}=: \lambda_2,
			\\
			&D^P_3(y,s)\geq \frac{11}{2}-4(1+O(\veps))\geq 1=: \lambda_3.
		\end{align*}
	\end{subequations}
For the moment, we let $\mu_j\ge0$ be the numbers such that 
	 $\|F_j^P(\cdot,s)\|_{L^{\infty}}\leq \mu_j, j=1,2,3$.
	 By integrating \eqref{Pder_eq} along  $\psi$, which is defined by $\frac{d}{ds}\psi(y;s)=\mathcal{U}$ and $\psi(y;s_0)=y$, we have
	\begin{equation*}\label{P'}
		\partial_y^jP(\psi(y;s),s) = \partial_y^jP(y,s_0)e^{-\int^s_{s_0}(D^P_j\circ\psi)\,ds'}+\int^s_{s_0}e^{-\int^s_{s'}(D^P_j\circ\psi)\,ds''}(F^P_j\circ\psi)\,ds'.
	\end{equation*}
Then in terms of $\lambda_j$ and $\mu_j$, one can obtain  the  bounds as 
	\begin{equation*}
		\|\partial_y^j P(\cdot,s)\|_{L^{\infty}} \leq \|\partial_y^j P(\cdot,s_0)\|_{L^{\infty}} 		e^{-\lambda_j(s-s_0)} + \frac{\mu_j}{\lambda_j},
	\end{equation*}
	which proves \eqref{EP2_Pder}.	

Now it is left to examine $F^P_j$: 		
	By \eqref{EP2_1D3-p},  \eqref{Wweak-p} and \eqref{d-t-e-p}, we see that
	\begin{equation*}
		\begin{split}
			\|F^P_1\|_{L^{\infty}}=\left\| \frac{U_{yy}P}{1-\dot{\tau}}+\frac{e^{-s}U_{yy}}{1-\dot{\tau}} \right\|_{L^{\infty}}
			& \leq 30\left(  \frac{A}{4}+\veps\right) =: \mu_1.
		\end{split}
	\end{equation*}
	Thus, we first get  $\|P_y(\cdot,s)\|_{L^{\infty}} \leq C_1$. With this bound for $P_y$ and  \eqref{EP2_1D3-p}, \eqref{EP2_1D5-p}, \eqref{d-t-e-p}, \eqref{EP2_1P2}, we have 
	\begin{equation*}
		\begin{split}
			\|F^P_2\|_{L^{\infty}}&=\left\|\frac{3U_{yy} P_y+\partial_y^3U P}{1-\dot{\tau}} + \frac{e^{-s}}{1-\dot{\tau}}\partial_y^3U \right\|_{L^{\infty}} 
			\\
			& \leq (1+O(\veps))\left( 45\|P_y\|_{L^{\infty}} + M^{5/6} \|P\|_{L^{\infty}} + e^{-s}M^{5/6} \right) 
			=: \mu_2,
		\end{split}
	\end{equation*}
	which gives  $\|P_{yy}(\cdot,s)\|_{L^{\infty}} \leq C_2$. Next we see that 
	\begin{equation*}
		\begin{split}
			\|F^P_3(\cdot,s)\|_{L^{\infty}}&\leq (1+O(\veps))\left(\| \partial_y^4U P\|_{L^{\infty}}+\|4\partial_y^3U P_y \|_{L^{\infty}}+\| 6U_{yy}P_{yy}\|_{L^{\infty}}+\|\partial_y^4Ue^{-s}\|_{L^{\infty}}\right)
			\\  
			&\leq (1+O(\veps))\left(M\|P(\cdot,s)\|_{L^{\infty}}+4M^{5/6}\|P_y(\cdot,s) \|_{L^{\infty}}+90\| P_{yy}(\cdot,s)\|_{L^{\infty}}+Me^{-s}\right)
			\\  
			& =: \mu_3.
		\end{split}
	\end{equation*}
	Here, the first and second inequalities hold thanks to   \eqref{EP2_1D3-p}, \eqref{EP2_1D5-p}, \eqref{EP2_1D4-p} and \eqref{d-t-e-p}. Also, for the last line, we used   \eqref{EP2_1P2} and the lower bounds for $\partial^j_y P$, $j=1,2$, that are obtained just previously in this proof.
	 This gives $\|\partial_y^3P(\cdot,s)\|_{L^{\infty}} \leq C_3$, which concludes the proof.	
\end{proof}

\begin{lemma}\label{Phider_lem}
	Under the assumptions of Proposition~\ref{mainprop}, there exists   a constant $C>0$ 
	such that for all $s\in[s_0,\sigma_1]$, 
	\begin{equation}\label{Pyes}
		\|\partial_y^{j}\Phi\|_{L^{\infty}}\leq  Ce^{-2s}. \qquad (j=2,3,4,5)
	\end{equation}
\end{lemma}

\begin{proof}
From \eqref{EP2_3-p} and   \eqref{EP2_3D-p}, we have 
\begin{subequations}
	\begin{align*}
		\Phi_{yy}e^s&=e^{\Phi-2s}-e^{-2s}-Pe^{-s},
		\\
		\partial_y^3\Phi e^s&=\Phi_ye^{\Phi-2s}-P_ye^{-s},
		\\
		\partial_y^4\Phi e^s&=(\Phi_{yy}+\Phi_y^2)e^{\Phi-2s}-P_{yy}e^{-s},
		\\
		\partial_y^5\Phi e^s&=(\partial_y^3\Phi+3\Phi_y\Phi_{yy}+\Phi_y^3)e^{\Phi-2s}-\partial_y^3Pe^{-s}.
	\end{align*}
\end{subequations}
 We see from \eqref{Phi_0_M} and \eqref{Phi_1_M-p} that 
\begin{equation*}
	\|e^{\Phi-2s}-e^{-2s}\|_{L^{\infty}}+\|\Phi_y e^{\Phi-2s}\|_{L^{\infty}}\leq Ce^{-2s}
\end{equation*} 
for some $C>0$.
Then thanks to \eqref{EP2_1P2} and \eqref{EP2_Pder},  we have  $\|\Phi_{yy}\|_{L^{\infty}}\lesssim e^{-2s}$ and $\|\partial_y^3\Phi \|_{L^{\infty}}\lesssim e^{-2s}$. Furthermore, using these bounds,
 we  also  see  that 
 \begin{subequations}
 	\begin{align*}
 	&\|(\Phi_{yy}+\Phi_y^2)e^{\Phi-2s}\|_{L^{\infty}}\leq Ce^{-2s},
 	\\
 	&\|(\partial_y^3\Phi+3\Phi_y\Phi_{yy}+\Phi_y^3)e^{\Phi-2s} \|_{L^{\infty}} \leq Ce^{-2s} 
 	\end{align*}
 \end{subequations} 
 for some $C>0$. 
These together with \eqref{EP2_1P2} and \eqref{EP2_Pder} yield the desired bounds for $\partial_y^4\Phi$ and $\partial_y^5\Phi$. 
This completes the proof of \eqref{Pyes}.
\end{proof}

\begin{lemma}\label{UW_far-p}
Under the assumptions of Proposition~\ref{mainprop}, it holds that   for all $s\in[s_0,\sigma_1]$,
	\begin{equation}
		|e^{s/2}(\kappa-\dot{\xi})|=\left|\frac{e^s\partial_y^3\Phi(0,s)}{\partial_y^3U(0,s)}\right|\leq Ce^{-s},
		\label{temp_1-p}
	\end{equation}
	\begin{equation}\label{xibound}
		|\kappa(t) |\leq 1, \qquad |\xi(t) |\leq 1,
	\end{equation}
	\begin{equation}\label{U_far}
		\inf_{\{|y|>1, s\in[s_0,\sigma_1]\}} \mathcal{U}(y,s) \frac{y}{|y|}  \geq \frac{1}{8},
	\end{equation}
	where $\mathcal{U}$ is defined in \eqref{U-p}, and $C>0$ is a constant.
\end{lemma}
\begin{proof}
	From \eqref{Eq_Modul2-p}, we see that 
	\begin{equation}\label{ax2}
		|e^{s/2}(\kappa-\dot{\xi})|=\left|\frac{e^s\partial_y^3\Phi(0,s)}{\partial_y^3U(0,s)}\right| \leq Ce^{-s},
	\end{equation}
	where we have used \eqref{EP2_1D2-p} and \eqref{Pyes} for the last inequality. 
	
	To prove \eqref{xibound}, we see from \eqref{Eq_Modul0-p} that there exists a  constant $C>0$ such that 
	\begin{equation*}
		|\dot{\kappa}|\leq \left|\frac{e^{3s/2}\partial_y^3\Phi(0,s)}{\partial_y^3U(0,s)}\right|+e^{3s/2}|\Phi_y(0,s)|\leq C e^{s/8},
	\end{equation*}
	 where we have used \eqref{Phi_1_M-p} and \eqref{ax2}. Therefore, we have
	\begin{equation}\label{kap_bound}
		\begin{split}
			|\kappa|&\leq \int^t_{-\veps} |\dot{\kappa}|\,dt'
			=\int^s_{s_0}|\dot{\kappa}|\frac{e^{-s'}}{1-\dot{\tau}}\,ds' 
			\leq C\int^s_{s_0}e^{-7s'/8}ds'
			 \leq 1.
		\end{split}
	\end{equation}
	Applying \eqref{ax2} and \eqref{kap_bound} to \eqref{Eq_Modul2-p}, we have that for all sufficiently small $\veps>0$,
	\begin{equation*}
		\begin{split}
			|\dot{\xi}|\leq \left|\frac{e^{s/2}\partial_y^3\Phi(0,s)}{\partial_y^3U(0,s)}\right|+|\kappa| \leq 2,
		\end{split}
	\end{equation*}
	from which we obtain the bound for $|\xi|$ in \eqref{xibound}. This finishes the proof of \eqref{xibound}.
	
	Using \eqref{Uy1-p},  \eqref{d-t-e-p} and \eqref{temp_1-p}, one can check that  $\mathcal{U}$ satisfies
	\begin{subequations}
		\begin{align*}
			&\mathcal{U}(y,s)\geq 	-\frac{y}{1-\dot{\tau}}+\frac{3y}{2}-\frac{Ce^{-s}}{1-\dot{\tau}}\geq \frac{1}{4}y-C\veps \quad \text{ for }y\geq 0,
			\\
			&\mathcal{U}(y,s)\leq 		-\frac{y}{1-\dot{\tau}}+\frac{3y}{2}+\frac{Ce^{-s}}{1-\dot{\tau}}\leq \frac{1}{4}y+C\veps \quad \text{ for }y<0.
		\end{align*}
	\end{subequations}
	This directly implies \eqref{U_far}.
\end{proof}

Next, we present Lemma~\ref{rmk2} giving the decaying properties of solutions to the transport-type equation, under suitable  assumptions. 
\begin{lemma} \label{rmk2}
Let  $f$ be a smooth solution to the equation 
\begin{equation*}
	\partial_s f(y,s) + \mathcal{U}(y,s)\partial_y f(y,s) +D(y,s)f(y,s) =F(y,s),
\end{equation*}  
where
$D$ and $F$ are smooth functions satisfying
\begin{subequations}
	\begin{align} 
		& \inf_{\{|y| \geq N,\, s\in[s_0,\infty)\}}D(y,s)\geq \lambda_D,\label{rmk2_ass1}
		\\
		& \|F (\cdot, s) \|_{L^{\infty}(|y| \geq N)}\leq F_0e^{-s\lambda_F}\label{rmk2_ass2}
	\end{align}
\end{subequations}
for some {$ \lambda_D, \lambda_F, N  \geq  0$}, and $\mathcal{U}$ is defined in \eqref{U-p}.   Then it holds that  
\begin{equation*}
	\begin{array}{l l}
		\limsup_{| y |\rightarrow \infty}|f(y,s)|\leq \limsup_{|y|\rightarrow \infty}{|f(y,s_0)|}e^{-\lambda_D(s-s_0)}+\frac{F_0}{\lambda_D-\lambda_F}e^{-s\lambda_F} & \quad \text{if } \lambda_D>\lambda_F, \\
		\limsup_{| y |\rightarrow \infty}|f( y ,s)|\leq \limsup_{|y|\rightarrow \infty}{|f(y,s_0)|}e^{-\lambda_D(s-s_0)}+\frac{F_0e^{-s_0\lambda_F}}{\lambda_F-\lambda_D}e^{-\lambda_D(s-s_0)} & \quad \text{if } \lambda_F>\lambda_D.
	\end{array}
\end{equation*}
\end{lemma}
We refer to \cite{BKK}, where  the proof of Lemma~$\ref{rmk2}$ is presented in the appendix.

\subsection{Proof of Proposition~\ref{mainprop} }\label{ch3.4}
This subsection is devoted  to prove Proposition~\ref{mainprop}. 
\subsubsection{The third derivative of $U$ at $y=0$.}
In Lemma~\ref{U30lem}, we close the bootstrap assumption \eqref{EP2_1D2-p}. 
\begin{lemma}\label{U30lem}
	 Under the assumptions of Proposition~\ref{mainprop}, we have that for all   $s\in[s_0,\sigma_1]$,
	\begin{equation}
		|\partial_y^3U(0,s)-6|\leq C\veps, \label{str_U3-p}
	\end{equation}
	where $C>0$ is a constant.
\end{lemma}

\begin{proof}
	Evaluating \eqref{EP2_2D3-p} at $y=0$ and using  \eqref{constraint-p}, \eqref{Eq_Modul2-p},  we obtain
	\begin{equation}\label{Eq_3rd}
		\partial_s \partial_y^3U^0 = \frac{4\dot{\tau}}{1-\dot{\tau}}\partial_y^3U^0+ \frac{\partial_y^4 U^0}{1-\dot{\tau}}\frac{e^s\partial_y^3 \Phi^0}{\partial_y^3 U^0}  -\frac{ e^s \partial_y^4\Phi^0 }{1-\dot{\tau}},
	\end{equation} 
	where $\partial_y^k f^0:=\partial_y^kf(0,s)$ for $f=U,\Phi$. 
	By \eqref{EP2_1D2-p}, \eqref{tau_dec} and \eqref{d-t-e-p}, the first term of the right-hand side of \eqref{Eq_3rd} is bounded as follows : 
	\begin{equation*}
		\left|\frac{4\dot{\tau}}{1-\dot{\tau}}\partial_y^3U^0\right|\leq 8|\dot{\tau}||\partial_y^3U^0|\leq Ce^{-s}.
	\end{equation*}
	For the remaining terms, using \eqref{EP2_1D4-p}, \eqref{d-t-e-p}, \eqref{Pyes} and \eqref{temp_1-p}, we obtain
	\begin{equation*}
		\left|\frac{\partial_y^4 U^0}{1-\dot{\tau}}\frac{e^s\partial_y^3 \Phi^0}{\partial_y^3 U^0} -\frac{ e^s \partial_y^4\Phi^0 }{1-\dot{\tau}}\right|\leq Ce^{-s}.
	\end{equation*} 
	Combining the above estimates, we obtain
	\begin{equation*} 
		\left|\partial_y^3U^0(s)-\partial_y^3U^0(s_0)\right|\leq 	\int_{s_0}^{s}{\left|\partial_{s'}\partial_y^3U^0(s')\right|\,ds'}\leq C\int_{-\log\varepsilon}^s {e^{-s'}\,ds'}\leq C\veps.
	\end{equation*}
	By \eqref{init_w_3-p}, we finish the proof.
\end{proof}

\subsubsection{The first derivative of $U$.}
In what follows,  $U_y$ near $y=0$ is estimated to close  \eqref{EP2_1D1-p}. 
Let $\wt{U} := U- \overline{U}$, where $\overline{U}$ is the blow-up profile of the Burgers equation solving \eqref{Burgers_SS}.  Then, from \eqref{Burgers_SS} and \eqref{EP2_2D1-p}, we see that 
\begin{equation}\label{Eq_diff-p}
	\partial_s \wt{U}_y + \left(1 + \frac{\wt{U}_y + 2\overline{U}'}{1-\dot{\tau}}\right) \wt{U}_y + \mathcal{U} \wt{U}_{yy}  = - \frac{e^s\Phi_{yy}}{1-\dot{\tau}} - \left( \frac{\wt {U} + \dot{\tau}\overline{U}}{1-\dot{\tau}}  + \frac{e^{s/2}(\kappa - \dot{\xi})}{1-\dot{\tau}} \right) \overline{U}'' - \frac{\dot{\tau}}{1-\dot{\tau}}(\overline{U}')^2.
\end{equation}
Note that the damping term, $1 + ( \wt{U}_y + 2\overline{U}')/(1-\dot{\tau})$, in \eqref{Eq_diff-p} is   strictly negative around $y=0$. 
To overcome this issue, we obtain a local behavior of  $\wt {U}_y$ around $y=0$ by employing a weight function, i.e., we set 
\[
V(y,s):=\frac{1+y^2}{y^2}\wt{U}_y(y,s).
\]
%
Then $V$ satisfies 
\begin{equation*}
	\partial_s V+D^V(y,s)V+\mathcal{U}(y,s)V_y=F^V(y,s)+\int_{\mathbb{R}}V(y',s)K^V(y,y',s)\,dy',
\end{equation*}
where $\mathcal{U}$ is defined in \eqref{U-p}, and
\begin{subequations}
	\begin{align*}
		D^V(y,s) &:= 1+\frac{\wt{U}_y+2\overline{U}'}{1-\dot{\tau}}+\frac{2}{y(1+y^2)}\mathcal{U},
		\\
		F^V(y,s) &:= - \frac{1+y^2}{y^2}\frac{ e^s \Phi_{yy}}{1-\dot{\tau}} - \frac{1+y^2}{y^2}\left( \frac{ \dot{\tau}\overline{U}}{1-\dot{\tau}}  + \frac{e^{s/2}(\kappa - \dot{\xi})}{1-\dot{\tau}} \right) \overline{U}'' - \frac{1+y^2}{y^2}\frac{\dot{\tau}}{1-\dot{\tau}}(\overline{U}')^2,
		\\
		K^V(y,s;y') &:= -\mathbb{I}_{[0,y]}(y')\frac{\overline{U}''(y)}{1-\dot{\tau}}\frac{1+y^2}{y^2}\frac{y'^2}{1+y'^2}.
	\end{align*}
\end{subequations}
%
\begin{lemma}\label{Wy1_lem_p}
Under the assumptions of Proposition~\ref{mainprop}, we have 
	\begin{equation}
		|U_y(y,s)-\overline{U}'(y)|\leq \frac{y^2}{20(1+y^2)} \label{Burgers_B.1-p}
	\end{equation}
		for all $y\in\mathbb{R}$ and $s\in[s_0,\sigma_1]$.
\end{lemma}

\begin{proof}
	We shall show that $\|V(\cdot,s)\|_{L^{\infty}(\mathbb{R})}\leq 1/20$. We first estimate $V$ near $y=0$. By expanding $\wt{U}_y$, we have
\begin{equation*}
	\wt{U}_y(y,s)=\frac{y^2}{2}(\partial_y^3U(0,s)-6)+\frac{y^3}{6}(\partial_y^4U(y',s)-\overline{U}^{(4)}(y')) \quad \text{for some }y'\in(-y,y)
\end{equation*}
	since $\wt{U}_y(0) =\wt{U}_{yy}(0) =0$ and $\overline{U}^{(3)}(0)=6$. Multiplying the above equation by $(1+y^2)/y^2$ and applying \eqref{EP2_1D4-p}, we get
	\begin{equation*}
		|V(y,s)|\leq \frac{1+y^2}{2}|\partial_y^3U(0,s)-6|+\frac{1+y^2}{6}|y|(M+\|\overline{U}^{(4)}\|_{L^{\infty}}).
	\end{equation*}
	Now, we choose sufficiently small $l=l(M)>0$ such that
	\begin{equation}\label{V-l-p}
		|V|\leq \frac{1+l^2}{2}|\partial_y^3U(0,s)-6|+\frac{1+l^2}{80} \leq \frac{1}{40}, \qquad  |y|\leq l
	\end{equation}
	for sufficiently small $\ve>0$. Here we used Lemma~\ref{U30lem}. 
	In what follows, we consider the region $|y|\geq l$.
%
%
%
%
%
%
We note that
\begin{equation*}
	\begin{split}
		D^V&=1+\frac{\wt{U}_y+2\overline{U}'}{1-\dot{\tau}}+\frac{2}{y(1+y^2)}\left(\frac{\overline{U}+\wt{U}}{1-\dot{\tau}}+\frac{3y}{2}+\frac{e^{s/2}(\kappa-\dot{\xi})}{1-\dot{\tau}}\right)
		\\
		&\geq \left(1+2\overline{U}'+\frac{2}{y(1+y^2)}\left(\frac{3y}{2}+\overline{U}\right)\right)-\left|\frac{\wt{U}_y+2\dot{\tau}\overline{U}'}{1-\dot{\tau}}+\frac{2}{y(1+y^2)}\left(\frac{\wt{U}+\dot{\tau}\overline{U}}{1-\dot{\tau}}+\frac{e^{s/2}(\kappa-\dot{\xi})}{1-\dot{\tau}}\right)\right|.
	\end{split}
\end{equation*}
We note that a part of $D^V$ has a non-negative lower bound as  
\begin{equation*}
	1+2\overline{U}'+\frac{2}{y(1+y^2)}\left(\frac{3y}{2}+\overline U\right)\geq \frac{y^2}{5(1+y^2)} + \frac{16y^2}{5(1+8y^2)}, \quad y \in \mathbb{R}, 
\end{equation*}
which is given in \eqref{0605_m_3} of Lemma~\ref{4.3}.

On the other hand, using  \eqref{EP2_1D1-p}, \eqref{tau_dec}, \eqref{d-t-e-p} and \eqref{y-w-y}, 
we have 
\begin{equation}\label{DV_1}
	\left|\frac{\wt{U}_y+2\dot{\tau}\overline{U}'}{1-\dot{\tau}}\right|\leq (1+O(\veps))\left(\frac{y^2}{10(1+y^2)}+C\veps\right)=\frac{y^2}{10(1+y^2)}+O(\veps).
\end{equation}
By  \eqref{EP2_1D1-p}, \eqref{tau_dec} and $\wt{U}(0,s)=0$, we have 
\begin{equation}\label{DV_2}
	\begin{split}
		\left|\frac{2}{y(1+y^2)}\frac{\wt{U}+\dot{\tau}\overline{U}}{1-\dot{\tau}}\right|&\leq \frac{2}{|y|(1+y^2)}\left(\frac{1}{10}\int^{|y|}_0\frac{y'^2}{(1+y'^2)}\,dy' +O(\veps)\right)
		\leq \frac{y^2}{15(1+y^2)}+O(\veps).
	\end{split}
\end{equation}
By \eqref{tau_dec}, \eqref{temp_1-p} and the fundamental theorem of calculus with $U(0,s)=0$,
\begin{equation}\label{DV_3}
	\begin{split}
		\left|\frac{2}{y(1+y^2)}\frac{e^{s/2}(\kappa-\dot{\xi})}{1-\dot{\tau}}\right|
		&\leq \frac{4}{|y|} |e^{s/2}(\kappa-\dot{\xi})| \leq  O(\veps), \qquad |y|\geq l.
	\end{split}
\end{equation}
Combining all, we have 
\begin{equation}\label{damp-p}
	\begin{split}
		D^V&\geq \left( \frac{y^2}{5(1+y^2)}+\frac{16y^2}{5(1+8y^2)} \right)-\frac{y^2}{10(1+y^2)}-\frac{y^2}{15(1+y^2)}-O(\veps)
		\\
		&\geq \frac{3y^2}{1+8y^2}+\frac{y^2}{30(1+y^2)}
		\ge
		\frac{3l^2}{1+8l^2}+\frac{l^2}{30(1+l^2)}, \qquad |y|\geq l.
	\end{split}
\end{equation}
 Here we used the fact that  the functions $y^2(1+8y^2)^{-1}$ and $y^2 ( 1+y^2)^{-1}$ are increasing functions on $y>0$.
 %

\emph{Estimate of $K^V$} : 
Using \eqref{d-t-e-p} and the property of $\overline{U}$ presented in \eqref{0605_m_4}, it is straightforward to check that 
\begin{equation}\label{kernel-p}
	\int_\mathbb{R}{|K^V(y,y',s)|\,dy'}\leq \frac{|\overline{U}''(y)|}{1-\dot{\tau}}\frac{1+y^2}{y^2}\int^{|y|}_0{\frac{(y')^2}{1+(y')^2 }\,dy'}\leq \frac{1}{\lambda}\left(\frac{3y^2}{1+8y^2}+\frac{y^2}{30(1+y^2)}\right)
\end{equation}
for some $\lambda>1$.
Thus, by \eqref{damp-p} and \eqref{kernel-p}, we deduce that 
\begin{equation}\label{KP-p}
	\int_\mathbb{R}{|K^V(y,y',s)|\,dy'}\leq \frac{1}{\lambda}D^V(y,s), \qquad  |y|\geq l.
\end{equation} 

\emph{Estimate of $F^V$} : 
Using \eqref{d-t-e-p}, we obtain 
\begin{equation}\label{FP-p}
	\begin{split}
		|F^V(\cdot,s)|
		&\leq C\frac{1+l^2}{l^2}\left(|e^s\Phi_{yy}|+|\dot{\tau}||\overline{U}\overline{U}''|+e^{s/2}|\kappa-\dot{\xi}||\overline{U}''|+|\dot{\tau}||\overline{U}'| \right)
		\\
		&\leq C\frac{1+l^2}{l^2}\left(e^{-s}(|y|+1)|\overline{U}''|+O(\veps)\right)
		\\
		&\leq \frac{1+l^2}{l^2} \cdot O(\veps), \qquad |y|\geq l.
	\end{split}
\end{equation}
Here, we have used   \eqref{tau_dec}, \eqref{Pyes}, \eqref{temp_1-p} and the properties of $\overline{U}$, i.e., \eqref{y-w-y},  \eqref{Wyybar_bdd} and the fact that
\begin{equation*}
	|\overline{U}(y)|\leq \int^{|y|}_0|\overline{U}'(y')|\,dy' \leq |y|,
\end{equation*}
thanks to \eqref{y-w-y2}. By choosing sufficiently small $\veps>0$, one can make $\|F^V(\cdot,s)\|_{L^{\infty}(|y|\geq l)}$ relatively small compared to the lower bound of $D^V$ in \eqref{damp-p}.


On the other hand, \eqref{4.3a-p} and \eqref{Utildey_M-p} imply 
\begin{equation}\label{Wy_claim}
	|V(y,s_0)|\leq \frac{1}{40}, \qquad \limsup_{|y|\rightarrow \infty}|V(y,s)|=0,
\end{equation} 
respectively.  
By applying Lemma~\ref{max_2} with \eqref{V-l-p}, \eqref{damp-p} and  \eqref{KP-p}--\eqref{Wy_claim}, we have  
\begin{equation*}
	\|V(\cdot,s)\|_{L^{\infty}(\mathbb{R})}\leq \frac{1}{20}.
\end{equation*}
This implies \eqref{Burgers_B.1-p}.
\end{proof}

\subsubsection{The second derivative of $U$.}
\begin{lemma}\label{Wy2_lem}
Under the assumptions of Proposition~\ref{mainprop},   we have
	\begin{equation*}
		|U_{yy}|\leq \frac{14|y|}{(1+y^2)^{1/2}}
	\end{equation*}
	for all $y\in\mathbb{R}$ and $s\in[-\log\veps,\sigma_1]$.
\end{lemma}

\begin{proof}
	
	We first estimate $U_{yy}$ near $y=0$. Expanding $U_{yy}$ at $y=0$, and then using \eqref{EP2_1D4-p} and \eqref{str_U3-p}, we get
	\begin{equation*}
		\begin{split}
			|U_{yy}(y,s)|&=|y||\partial_y^3U(0,s)|+\frac{y^2}{2}|\partial_y^4U(y',s)|  \qquad \text{for some }|y'|\in (0,|y|)
			\\
			&\leq |y|\left(6+C\veps+\frac{M}{2}|y|\right).
		\end{split}
	\end{equation*}
	Now we choose sufficiently large $M>0$ and sufficiently small $\veps_0>0$ such that for all $\veps\in (0,\veps_0)$, it holds that
	\begin{equation}\label{Uyy_local}
		|U_{yy}(y,s)|\leq \frac{7|y|}{(1+y^2)^{1/2}} \qquad \text{as long as }|y|\leq \frac{1}{M}.
	\end{equation}

	In what follows, we obtain some estimates on the region $|y|\geq 1/M$ and apply Lemma~\ref{max_2}. Define $\wt{V}(y,s):= (1+y^2)^{1/2}{y}^{-1}U_{yy}(y,s)$. From \eqref{EP2_2D2-p}, $\wt{V}(y,s)$ satisfies
	\begin{equation*}
			\partial_s\wt{V} + D^{\wt{V}} \wt{V} + \mathcal{U} \wt{V}_y=F^{\wt{V}},
	\end{equation*}
	where $\mathcal{U}$ is given in \eqref{U-p}, and
	\begin{subequations}
		\begin{align*}
			D^{\wt{V}} &:= \frac{5}{2}+\frac{3U_y}{1-\dot{\tau}}+\frac{1}{y(1+y^2)}\left(\frac{U}{1-\dot{\tau}}+\frac{3}{2}y+\frac{e^{s/2}(\kappa-\dot{\xi})}{1-\dot{\tau}}\right),
			\qquad 
			F^{\wt{V}} := -\frac{e^s\partial_y^3\Phi}{1-\dot{\tau}}\frac{(1+y^2)^{1/2}}{y}.
		\end{align*}
	\end{subequations}

\emph{Lower bound of $D^{\wt{V}}$} : 
	Rearranging the terms in $D^{\wt{V}}$, we see that
	\begin{equation*}
		\begin{split}
			D^{\wt{V}}
			&=\left(\frac{5}{2}+3\overline{U}'+\frac{1}{1+y^2}\left(\frac{3}{2}+\frac{\overline{U}}{y}\right)\right)
			\\
			&\quad +\frac{3}{1-\dot{\tau}}((U_y-\overline{U}')+\dot{\tau}\overline{U}')+\frac{1}{y(1+y^2)}\frac{U-\overline{U}}{1-\dot{\tau}}+\frac{1}{y(1+y^2)}\left(\frac{\dot{\tau}\overline{U}}{1-\dot{\tau}}+\frac{e^{s/2}(\kappa-\dot{\xi})}{1-\dot{\tau}}\right).
		\end{split}
	\end{equation*}
	Note that by \eqref{y-w-y2} and \eqref{0605_m_3}, we obtain for $y\in\mathbb{R}$, 
		\begin{equation*}
		\begin{split}
			\frac{5}{2}+3\overline{U}' +\frac{1}{1+y^2}\left(\frac{3}{2}+\frac{\overline{U}}{y}\right)
			&= 2(1+\overline{U}')+\frac{1}{2}\left(1+2\overline{U}'+\frac{2}{1+y^2}\left(\frac{3}{2}+\frac{\overline{U}}{y}\right)\right) 
			\\
			&\geq \frac{2y^2}{3(1+y^2)}+\frac{1}{2}\left(\frac{y^2}{5(1+y^2)}+\frac{16y^2}{5(1+8y^2)}\right) 
			\geq \frac{2y^2}{3(1+y^2)}.
		\end{split}
	\end{equation*}
	Here, we used the fact that 
	\begin{equation*}
		1+\overline{U}'\geq 1-\frac{1}{1+\frac{3y^2}{(3y^2+1)^{2/3}}}
		\geq \frac{y^2}{3(1+y^2)},
	\end{equation*} 
	which is true by \eqref{y-w-y2}.  One can show that the remaining terms are sufficiently small, which implies that   $D^{\wt{V}}$ has a positive lower bound. As in the proof of Lemma~\ref{Wy1_lem_p} (see \eqref{DV_1}--\eqref{DV_3}), we find that for sufficiently small   $\ve>0$, 
	
	\begin{equation}\label{Uyy_D}
		\begin{split}
			D^{\wt{V}}	&\geq \frac{2y^2}{3(1+y^2)}-\frac{3y^2}{10(1+y^2)}-\frac{y^2}{30(1+y^2)}-O(\veps)
			\geq \frac{y^2}{4(1+y^2)} \geq \frac{1}{4(M^2+1)}=:\lambda_D
		\end{split}
	\end{equation}
	uniformly in $|y|\geq 1/M$. 

	\emph{Estimate of $F^{\wt{V}}$}:
	Thanks to \eqref{d-t-e-p} and \eqref{Pyes}, we have for sufficiently small $\veps$ that 
	\begin{equation}\label{Uyy_F2}
		\|F^{\wt{V}}\|_{L^{\infty}(|y|\geq \frac{1}{M})} \leq CM\left(1+\frac{1}{M^2}\right)^{1/2}e^{-s}
		\leq CM\left(1+\frac{1}{M^2}\right)^{1/2}\veps
	\end{equation}
for some constant $C>0$.

	\emph{Asymptotic behavior of $\wt{V}$ at $|y|=\infty$} : 
	Finally, we claim that 
	\begin{equation}\label{wtV_dec}
		\limsup_{|y|\rightarrow \infty}|\wt{V}|<14.
	\end{equation} Thanks to \eqref{Uyy_D} and \eqref{Uyy_F2}, one has
	\begin{equation*}
		\inf_{|y|\geq 1/M}D^{\wt{V}}(y,s)\geq \lambda_D(M), \qquad \|F^{\wt{V}}(y,s)\|_{L^{\infty}(|y|\geq 1/M)}\leq 	F_0(M)e^{-\lambda_Fs},
	\end{equation*}
	where $1=\lambda_F>\lambda_D>0$ for sufficiently large $M$. Then, we apply Lemma~\ref{rmk2} to get
	\begin{equation*}
		\limsup_{|y|\rightarrow \infty}|\wt{V}(y,s)|\leq \limsup_{|y|\rightarrow \infty}|\wt{V}(y,s_0)|+C(M)\veps\le 7 + C(M)\ve, 
	\end{equation*}
	where we have used \eqref{init_24-p} for the last inequality. 
	This implies  \eqref{wtV_dec} for sufficiently small $\ve>0$.
%
	Then we apply Theorem~\ref{max_2} with \eqref{init_24-p}, \eqref{Uyy_local} and \eqref{Uyy_D}-\eqref{wtV_dec} to finish the proof of Lemma~\ref{Wy2_lem}.
\end{proof}

\subsubsection{The third derivative of $U$.}
\begin{lemma}\label{w-3-est}
Under the assumptions of Proposition~\ref{mainprop}, we have
	\begin{equation*}
		\|\partial_y^3U(\cdot,s)\|_{L^{\infty}}\leq \frac{M^{5/6}}{2} 
	\end{equation*}
	for $s\in[s_0,\sigma_1]$.
\end{lemma}

\begin{proof}
	Let us first consider the region near $y=0$. Expanding $\partial_y^3U$ at $y=0$, and then using \eqref{EP2_1D2-p} and \eqref{EP2_1D4-p}, we have
	\begin{equation} \label{Uy3loc}
		|\partial_y^3U(y,s)|\leq 	|\partial_y^3U(0,s)|+|y|\|\partial_y^4U(\cdot,s)\|_{L^{\infty}}\leq \frac{M^{5/6}}{4}
	\end{equation}
 for $|y|\leq {M^{-1/6}/8}$ with sufficiently large $M>0$ such that $7\leq  M^{5/6}/{8}$.

	Recalling \eqref{EP2_2D3-p}, we see that
	\begin{equation*}\label{Uyyyeq}
		\partial_s \partial_y^3U + D^U_3 \partial_y^3U + 	\mathcal{U} \partial_y^4 U  =  F^U_3,
	\end{equation*}
	where $\mathcal{U}$ is defined in \eqref{U-p}, and 
	\begin{subequations}
		\begin{align*}
			D^U_3 &:= 4 + \frac{4U_y}{1-\dot{\tau}} ,\qquad 
			F^U_3 := - \frac{3U_{yy}^2}{1-\dot{\tau}} -\frac{ e^s \partial_y^4\Phi}{1-\dot{\tau}}.
		\end{align*}
	\end{subequations}

	\emph{Lower bound of $D^U_3$}:
	Note that by \eqref{y-w-y2}, we have
	\begin{equation*}
		1+\overline{U}'\geq 1-\frac{1}{1+\frac{3y^2}{(3y^2+1)^{2/3}}}\geq \frac{y^2}{3(1+y^2)}.
	\end{equation*}
	Using this, \eqref{EP2_1D1-p}, \eqref{Uy1-p} and \eqref{tau_dec},  we have 
	\begin{equation*}
		\begin{split}
			D^U_3&= 4(1+\overline{U}')+4(U_y-\overline{U}')+\frac{4\dot{\tau}U_y}{1-\dot{\tau}}
			\geq \frac{4y^2}{3(1+y^2)}-\frac{4y^2}{10(1+y^2)}-O(\veps).
		\end{split}
	\end{equation*} 
	One can choose  $\ve_0=\veps_0(l)>0$ such that for all $\veps\in (0,\veps_0)$, it holds that 
	\begin{equation}\label{Uy3D'}
		D^U_3\geq \frac{y^2}{2(1+y^2)}\geq \frac{l^2}{2(1+l^2)}
	\end{equation} 
	as long as $|y|>l>0$. In particular, we have that for $|y|>\frac{1}{8M^{1/6}}$,
	\begin{equation*}
		D^U_3\geq \frac{1}{2(64M^{1/3}+1)}.
	\end{equation*}

	\emph{Estimate of $F^U_3$} : 
	By \eqref{EP2_1D3-p}, \eqref{d-t-e-p} and \eqref{Pyes}, 
	\begin{equation}\label{Uy3F}
		\begin{split}
			\|F^U_3\|_{L^{\infty}}&\leq 	(1+O(\veps))(3\|U_{yy}\|_{L^{\infty}}^2+\|e^s\partial_y^4\Phi\|_{L^{\infty}})
			\leq (1+O(\veps))(3\cdot 15^2+Ce^{-s})
			\leq 4\cdot 15^2.
		\end{split}
	\end{equation}

	\emph{Asymptotic behavior of $\partial_y^3U$ at $|y|=\infty$} : 
	Next, we claim 
	\begin{equation}\label{Wy3_dec}
		\limsup_{|y|\rightarrow \infty}|\partial_y^3U(y,s)|< \frac{M^{5/6}}{2}.
	\end{equation} 
	As we have checked in \eqref{Uy3D'} and \eqref{Uy3F}, we have 
	\begin{equation*}
		\inf_{|y|>1}D^U_3\geq 1/2, \qquad \|F^U_3\|_{L^{\infty}\left(|y|>1\right)}\leq 4\cdot 15^2.
	\end{equation*} 
	 Applying   Lemma~\ref{rmk2}, we have
	\begin{equation*}
		\limsup_{|y|\rightarrow \infty}|\partial_y^3U(y,s)|\leq 	\limsup_{|y|\rightarrow \infty}|\partial_y^3U(y,s_0)|e^{-(s-s_0)/2}+8\cdot 15^2 \leq M^{5/6}/4 
	\end{equation*}
	for sufficiently large $M>0$. Here, we used \eqref{init_24-p} for the last inequality.
	Hence, applying Theorem~\ref{max_2} with \eqref{init_24-p}, \eqref{Uy3loc} and  \eqref{Uy3D'}-\eqref{Wy3_dec},  we finish the proof.
\end{proof}

\subsubsection{The fourth derivative of $U$.}
\begin{lemma}\label{Wy4_lem-p}
Under the assumptions of Proposition~\ref{mainprop}, we have
	\begin{equation*}
		\|\partial_y^4U(\cdot,s)\|_{L^{\infty}}\leq \frac{M}{2}
	\end{equation*}
	for all $s\in[s_0,\sigma_1]$.
\end{lemma}

\begin{proof}
	From \eqref{EP2_2D4-p}, we have
	\begin{equation}\label{U4-y}
		\partial_s \partial_y^4U + D^U_4 \partial_y^4U + \mathcal{U} \partial_y^5 U  = F^U_4, 
	\end{equation}
	where $\mathcal{U}$ is defined in \eqref{U-p}, and
	\begin{subequations}
		\begin{align*}
			D^U_4 &:= \frac{11}{2} + \frac{5U_y}{1-\dot{\tau}}, \qquad F^U_4 := - \frac{10U_{yy}\partial_y^3U}{1-\dot{\tau}} -\frac{ e^s\partial_y^5\Phi}{1-\dot{\tau}}.
		\end{align*}
	\end{subequations}
	By \eqref{Uy1-p} and \eqref{tau_dec}, we obtain 
	\begin{equation}\label{DU4}
		D^U_4 \geq \frac{11}{2}-\frac{5|U_y|}{1-\dot{\tau}}\geq \frac{11}{2}-5-O(\veps)\geq \frac{1}{4}.
	\end{equation}
	On the other hand, thanks to  \eqref{EP2_1D3-p}, \eqref{EP2_1D5-p}, \eqref{d-t-e-p} and \eqref{Pyes}, it is straightforward to check that
	\begin{equation}\label{FU4}
		\begin{split}
			\|F^U_4(\cdot,s)\|_{L^{\infty}}
			\leq (1+O(\veps))(150M^{5/6}+1).
		\end{split}
	\end{equation}
	
	Let $\psi(y;s)$ be the characteristic curve along $\mathcal{U}$, i.e., $\frac{d}{ds} \psi(y;s)=\mathcal{U}(\psi(y;s),s)$, $\psi(y;s_0)=y$. 
	By integrating \eqref{U4-y} along $\psi$, we have 
	\begin{equation*}
		\partial_y^4U(\psi(y;s),s)=\partial_y^4U(y,s_0)e^{-\int^s_{s_0}(D^U_4\circ\psi)\,ds'}+\int^s_{s_0}e^{-\int^s_{s'}(D^U_4\circ\psi)\,ds''}(F^U_4\circ\psi)\,ds', 
	\end{equation*}
	for which we use \eqref{init_24-p}, \eqref{DU4} and \eqref{FU4} to get
	\begin{equation*}
		\|\partial_y^4U\|_{L^{\infty}}\leq 	\|\partial_y^4U(\cdot,s_0)\|_{L^{\infty}}e^{-\frac{1}{4}(s-s_0)}+4(1+O(\veps))(150M^{5/6}+1)\leq \frac{M}{2}
	\end{equation*}
	for sufficiently large $M>0$. This completes the proof.
\end{proof}

\subsubsection{Weighted estimate of $\rho-1$.}
%
Here we establish the weighted estimate for $P$, which plays a crucial role in obtaining the blow-up profile of $\rho-1$. 
\begin{lemma}\label{propW}
Under the assumptions of Proposition~\ref{mainprop},  we have
	\begin{equation*}
		|P(y,s)| <  \frac{A}{y^{2/3}+8}
	\end{equation*}
for all $y\in\mathbb{R}$ and $s\in[s_0,\sigma_1]$.
\end{lemma}

\begin{proof}
Setting $\wt{P} := (y^{2/3}+8)P$, 
we see from \eqref{EP2_1-p} that $\wt{P}$ satisfies 
\begin{equation}\label{Weq}
	\partial_s \wt{P} + D^{\wt{P}}\wt{P} + \mathcal{U}\wt{P}_y = F^{\wt{P}},
\end{equation}
where $\mathcal{U}$ is defined in \eqref{U-p} and
\begin{subequations}
	\begin{align*}
		D^{\wt{P}} &:= 1+U_y-\frac{2}{3}\frac{y^{2/3}}{y^{2/3}+8}\left(\frac{3}{2}+\frac{U}{y}\right),
		\\
		F^{\wt{P}} &:= -\frac{e^{-s}}{1-\dot{\tau}}U_y(y^{2/3}+8)+\left(-\frac{\dot{\tau}}{1-\dot{\tau}}U_y+\frac{2}{3}\frac{y^{2/3}}{(y^{2/3}+8)}\left(\frac{\dot{\tau}}{1-\dot{\tau}}\frac{U}{y}+\frac{1}{y}\frac{e^{s/2}(\kappa-\dot{\xi})}{1-\dot{\tau}}\right)\right)\wt{P}.
	\end{align*}
\end{subequations}
	We shall show that
		\begin{equation}\label{claim_W}
		|\wt{P}(y,s)| < A.
	\end{equation}
	First, by \eqref{Utildey_M-p} and \eqref{0605_m_7},  we see that  
	\begin{equation}\label{W_D}
		\begin{split}
			D^{\wt{P}}(y,s)&=1+(U_y-\overline{U}')+\overline{U}'-\frac{2y^{2/3}}{3(y^{2/3}+8)}\left(\frac{3}{2}+\frac{\overline{U}}{y}+\frac{U-\overline{U}}{y}\right)
			\\
			&\geq 1-\frac{1}{y^{2/3}+8}+\overline{U}'-\frac{2y^{2/3}}{3(y^{2/3}+8)}\left(\frac{3}{2}+\frac{\overline U}{y}+\frac{1}{y}\int_{0}^{y}{\frac{dy'}{{y'}^{2/3}+8}}\right) 
			\\
			&=: D^{\wt{P}}_{-}(y)\geq 0 \quad \text{for}\quad|y|\geq 3.
		\end{split}
	\end{equation}
	Next, we shall check that 
	\begin{equation}\label{W_F}
		\|F^{\wt{P}}(y,s)\|_{L^{\infty}(|y|\geq 3)}\leq C e^{-s}.
	\end{equation}
By \eqref{Utildey_M-p}, \eqref{d-t-e-p} and the fact that $|(y^{2/3}+8)\overline{U}'|\leq C$ from \eqref{y-w-y2}, we obtain 
	\begin{equation}\label{F_Pyy_2}
		\begin{split}
			\left|\frac{e^{-s}}{1-\dot{\tau}}U_y(y^{2/3}+8)\right|&\leq (1+O(\veps))e^{-s}(\|(y^{2/3}+8)\overline{U}'\|_{L^{\infty}}+\|(y^{2/3}+8)(U_y-\overline{U}')\|_{L^{\infty}})
			 \leq Ce^{-s}.
		\end{split}
	\end{equation}
Note from \eqref{constraint-p} and \eqref{Uy1-p}  that
  \begin{equation*}
 | U(y)|   \leq |y|.
  \end{equation*}
Using this and \eqref{Wweak-p},    \eqref{Uy1-p}, \eqref{tau_dec}, \eqref{d-t-e-p} and \eqref{temp_1-p}, we have  
\begin{equation*}
	\begin{split}
	\left|-\frac{\dot{\tau}}{1-\dot{\tau}}U_y+\frac{2}{3}\frac{y^{2/3}}{(y^{2/3}+8)}\left(\frac{\dot{\tau}}{1-\dot{\tau}}\frac{U}{y}+\frac{1}{y}\frac{e^{s/2}(\kappa-\dot{\xi})}{1-\dot{\tau}}\right)\right| |\wt{P}| 
	&\leq C e^{-s} 
	\end{split}
\end{equation*}
for some $C>0$. 
 This and \eqref{F_Pyy_2} imply \eqref{W_F}.

	
	\vspace{1\baselineskip}
	
	Now we claim that \eqref{claim_W} holds.
	Suppose to the contrary that it fails.  Then $s_2 := \min\{s\in [s_0,\sigma_1] : \|\wt{P}(\cdot,s)\|_{L^{\infty}}=A\}$ is well-defined. We also see that there is a number $s_1\in(s_0,s_2)$ such that
	\begin{equation}\label{2.45}
		\frac{3A}{4} = \|\wt{P}(\cdot,s_1)\|_{L^{\infty}} \leq \|\wt{P}(\cdot,s)\|_{L^{\infty}} < A \qquad \text{for all } s\in(s_1,s_2)
	\end{equation}
	 since $\wt{P}\in C([s_0,\sigma_1]; L^{\infty}(\mathbb{R}))$ and $\|\wt{P}(\cdot,s_0)\|_{L^{\infty}} \leq A/2$ from \eqref{init_ttP-p}.

	First we note that  
	\begin{equation}\label{2.40}
		\limsup_{|y|\rightarrow \infty}|\wt{P}(y,s) |<\frac{3A}{4}. 
	\end{equation} 
 To see this, 
	 we apply Lemma~\ref{rmk2} with \eqref{W_D} and \eqref{W_F}, i.e., $\inf_{|y|\geq 3}D^{\wt{P}}(y,s)\geq 0$ and $\|F^{\wt{P}}(y,s)\|_{L^{\infty}(|y|\geq 3)}\leq Ce^{-s}$, respectively, to get 
	\begin{equation*}
		\limsup_{|y|\rightarrow \infty}|\wt{P}(y,s)|\leq \limsup_{|y|\rightarrow \infty}|\wt{P}(y,s_0)|+C\veps \le \frac{A}{2} + C\ve, 
	\end{equation*}
	where we used \eqref{init_ttP-p} for the last inequality. Thus, \eqref{2.40} holds for small $\ve>0$.

	Then, for each $s\in[s_1,s_2]$, thanks to the fact that   $\wt{P}$ is smooth and decay as $|y|\rightarrow \infty$ by \eqref{2.40},  we see that $C^{\wt{P}}:=\{ y \in \mathbb{R} :  \partial_y \wt{P}(y,s) =0 \}$  is a non-empty closed set. Hence,   one can choose a point $y_*(s)\in C^{\wt{P}} $ such that  
	\begin{equation*}
		\|\wt{P}(\cdot, s) \|_{L^{\infty}} = |\wt{P}(y_*(s), s)|.
	\end{equation*} 
	By the definition of $y_*(s)$ and the choice of $s_1$ in \eqref{2.45},  we see that $\wt{P}(y_*(s), s) \neq 0$ for all $s \in [s_1,s_2]$.
	For any fixed $s$, it holds that $\|\wt{P}(\cdot,s-h)\|_{L^\infty} \geq |\wt{P}(y_*(s) - h \mathcal{U}(y_*(s),s), s - h)| $ for any small $h>0$,  by the definition of $y_*$. 
	Then, it is straightforward to check that 
	\begin{equation}\label{AP_R1-p} 
		\begin{split}
			\lim_{h \to 0^+} \frac{\|\wt{P}(\cdot,s-h)\|_{L^\infty} - \|\wt{P}(\cdot,s)\|_{L^\infty}}{-h} 
			& \leq (\partial_s + \mathcal{U}(y_*(s),s)\partial_y)|\wt{P}|(y_*(s),s)
		\end{split}
	\end{equation} 
	provided that the limit on the LHS of \eqref{AP_R1-p} exists. Note that by Rademacher's theorem, 
	$ \| \wt{P}(\cdot, s)\|_{L^{\infty}}$, being Lipschitz, is differentiable at almost all $s\in[s_1, s_2]$.  From \eqref{AP_R1-p} and the fact that $\partial_y \wt{P}(y_*(s), s) =0$, we see that 
	\begin{equation} \label{wtP-thm-p}
		\begin{split}
			\frac{d}{ds} \| \wt{P}(\cdot, s)\|_{L^{\infty}} 
			& \leq  \left\{ \begin{array}{l l}
				\partial_s \wt{P}(y, s)|_{y=y_*(s)} & \text{if } \wt{P}(y_*(s),s)>0, \\
				-\partial_s \wt{P}(y, s)|_{y=y_*(s)} & \text{if } \wt{P}(y_*(s),s)<0. 
			\end{array}
			\right.  
		\end{split}
	\end{equation} 
	
	On the other hand, we see from \eqref{EP2_1P2}  that 
	\begin{equation*}
		|P(y,s)|\leq (1+C\veps)\|P(\cdot,s_0)\|_{L^{\infty}}+C\veps=(1+C\veps)\veps\|\rho_0-1\|_{L^{\infty}}+C\veps,
	\end{equation*}
	which together with \eqref{init_ttP-p}  yields 
	\begin{equation*}
		|\wt{P}(y,s)|=(y^{2/3}+8)|P(y,s)|\leq (3^{2/3}+8)\frac{A}{16}+C\veps < \frac{3A}{4} \quad  \text{for } |y|\leq 3.
	\end{equation*}
	This together with \eqref{2.45} and the definition of $y_*(s)$ implies that  
	\begin{equation}\label{y*}
		|y_*(s)|\geq 3\qquad \text{for all } s\in[s_1,s_2].
	\end{equation}
	
	For any $s$, suppose that $\wt{P}(y_*(s),s)>0$ holds. In this case, thanks to \eqref{W_D} and \eqref{y*}, it holds that 
	\begin{equation}\label{W_D'}
		D^{\wt{P}}(y_*(s),s)\wt{P}(y_*(s),s)\geq D^{\wt{P}}_{-}(y_*(s))\|\wt{P}(\cdot,s)\|_{L^{\infty}}\geq 0.
	\end{equation}
 In view of \eqref{Weq} with \eqref{W_F}, \eqref{W_D'} and the fact that $\partial_y\wt{P}(y_*(s),s)=0$, we find 
	\begin{equation}\label{wtP_pos}
		\partial_s\wt{P}(y_*(s),s)\leq  Ce^{-s}-D^{\wt{P}}(y_*,s)\wt{P}(y_*,s)\leq C e^{-s}.
	\end{equation}
	Similarly, for the case that $\wt{P}(y_*(s),s)<0$, we have
	\begin{equation}\label{wtP_neg}
		-\partial_s\wt{P}(y_*,s)\leq C e^{-s}+D^{\wt{P}}(y_*,s)\wt{P}(y_*,s) \leq Ce^{-s}.
	\end{equation}
	Then we integrate \eqref{wtP-thm-p}, using  \eqref{wtP_pos}, \eqref{wtP_neg}
	and the fact that $\|\wt{P}(\cdot,s_1)\|_{L^{\infty}}=3A/4$, 
	to have
	\begin{equation*}
		\begin{split}
			\|\wt{P}(\cdot,s)\|_{L^{\infty}}&\leq \frac{3A}{4} +C \int^s_{s_1}e^{-s'}\,ds'  \leq \frac{3A}{4} +C\veps <A
		\end{split}
	\end{equation*}
	for all $s\geq s_1$ provided that $\ve>0$ is sufficiently small. 
	This  contradiction 
	  completes the proof.
	\end{proof}

\subsubsection{Weighted estimate of $U_y$.}
Our main goal of this section is to prove Lemma~\ref{mainprop_1-p}. For this, we first remark that the assumption \eqref{Wweak-p} implies 
\begin{equation}
	\|(y^{2/3}+8)\Phi_{yy}\|_{L^{\infty}}\leq Ce^{-2s}\label{Phiyy_decay-p}
\end{equation}
for a positive constant $C$.
To see this, we consider the change of variable $\tilde{y} = e^{-3s/2}y$. Using this and \eqref{EP2_3-p}, we have
	\[
	-\partial_{\tilde{y}\tilde{y}}\wt{\Phi} = f(\tilde{y},s) + 1- e^{\wt{\Phi}},
	\]
	where $\wt{\Phi}(\tilde{y},s) = \Phi(y,s)$ and $f(\tilde{y},s) = P(y,s)e^s$. Thanks to \eqref{Wweak-p}, we also see that
	\begin{equation*}
		 |\tilde{y}^{2/3} f(\tilde{y},s)|\leq \|e^{-s}(\tilde{y}^{2/3}e^s + 8)f(\tilde{y},s)\|_{L^{\infty}_{\tilde y}}\leq 2A.
	\end{equation*}
	Then, applying Lemma~\ref{Lem_Pois}, together with \eqref{Aeq}, we obtain 
	\begin{equation*}
		C \geq |(\tilde{y}^{2/3}+1)\partial_{\tilde{y}}^2\wt{\Phi}| = |(y^{2/3}+e^s)e^{2s}\Phi_{yy}|\geq |(y^{2/3}+8)e^{2s}\Phi_{yy}|.
	\end{equation*}
	This implies \eqref{Phiyy_decay-p}.

Next, we present an auxiliary lemma which will be used in the proof of Lemma~\ref{mainprop_1-p}.
\begin{lemma}
Under the  assumptions of Proposition~\ref{mainprop}, we have
	\begin{equation}\label{Ut_dec_p}
		\limsup_{|y|\rightarrow \infty} |(y^{2/3}+8)(U_y-\overline{U}')|<\frac{23}{25} 
	\end{equation}
	for all $s\in[s_0,\sigma_1]$.
\end{lemma}
\begin{proof}
	We first note that 
	\begin{equation*}
		\begin{split}
			\limsup_{|y|\rightarrow \infty}|(y^{\frac{2}{3}}+8)(U_y-\overline{U}')|&\leq \limsup_{|y|\rightarrow \infty}|(y^{2/3}+8)U_y|+\lim_{|y|\rightarrow \infty}|(y^{\frac{2}{3}}+8)\overline{U}'|
			\\
			&= \limsup_{|y|\rightarrow \infty}|(y^{2/3}+8)U_y|+\frac{1}{3}. 
		\end{split}
	\end{equation*}
	Thus, it is enough to show that
	\begin{equation*}\label{3.00}
		\limsup_{|y|\rightarrow \infty} |(y^{2/3}+8)U_y(y,s)|<\frac{44}{75}.
	\end{equation*}
	Letting $\mu(y,s):=(y^{2/3}+8)U_y(y,s)$, we have from \eqref{EP2_2D1-p} that
	\begin{equation*}
		\partial_s \mu+D^{\mu}\mu+\mathcal{U}\mu_y=F^{\mu}, 
	\end{equation*}
	where $\mathcal{U}$ is defined in \eqref{U-p}, and
	\begin{subequations}
		\begin{align*}
			D^{\mu} &:= 1+U_y-\frac{2y^{2/3}}{3(y^{2/3}+8)}\left(\frac{3}{2}+\frac{U}{y}\right),
			\\
			F^{\mu} &:= -\frac{e^s \Phi_{yy}}{1-\dot{\tau}}(y^{2/3}+8)+\left(-\frac{\dot{\tau}U_y}{1-\dot{\tau}}+\frac{2}{3}\frac{y^{2/3}}{y^{2/3}+8}\left(\frac{\dot{\tau}}{1-\dot{\tau}}\frac{U}{y}+\frac{1}{y}\frac{e^{s/2}(\kappa-\dot{\xi})}{1-\dot{\tau}}\right)\right)\mu.
		\end{align*}
	\end{subequations}
	Using \eqref{Utildey_M-p} and the fundamental theorem of calculus with $(U-\overline{U})(0,s)=0$, we have
	\begin{equation*}
		\begin{split}
			D^{\mu}(y,s)&=1+(U_y-\overline{U}')+\overline{U}'-\frac{2y^{2/3}}{3(y^{2/3}+8)}\left(\frac{3}{2}+\frac{\overline{U}}{y}+\frac{U-\overline{U}}{y}\right)
			\\
			&\geq 1-\frac{1}{y^{2/3}+8}+\overline{U}'-\frac{2y^{2/3}}{3(y^{2/3}+8)}\left(\frac{3}{2}+\frac{\overline U}{y}+\frac{1}{y}\int_{0}^{y}{\frac{dy'}{{y'}^{2/3}+8}}\right). 
		\end{split}
	\end{equation*}
	Thanks to \eqref{0605_m_7} and $\overline{U}'<0$, we see that the lower bound of $D^\mu$ is strictly positive for $|y|\geq 3$, 
	i.e.,  \begin{equation}\label{Dmu}
	\inf_{|y|\geq 3}D^{\mu} (y,s) \geq \lambda^\mu_D >0
	\end{equation} 
	for some $\lambda^\mu_D>0$. 
	From \eqref{Utildey_M-p} and \eqref{y-w-y}, we have
	\begin{equation}\label{mu}
		|\mu|\leq |(y^{2/3}+8)(U_y-\overline{U}')|+|(y^{2/3}+8)\overline{U}'|\leq C.
	\end{equation}
 Using \eqref{Uy1-p}, \eqref{tau_dec}, \eqref{d-t-e-p},    \eqref{temp_1-p} and \eqref{Phiyy_decay-p}, we can check that
\begin{equation}\label{Fmu}
	\begin{split}
			\|F^{\mu}(y,s)\|_{L^{\infty}(|y|\geq 3)} &\leq C(e^{-s}+e^{-s}\|\mu\|_{L^{\infty}})
			 \leq Ce^{-s},
	\end{split}
\end{equation}
where we have used \eqref{mu} for the last inequality. 
	 Now  applying  Lemma~\ref{rmk2} with \eqref{Dmu} and \eqref{Fmu}, we get
	\begin{equation*}
		\limsup_{|y|\rightarrow \infty}|\mu(y,s)|\leq \limsup_{|y|\rightarrow \infty}|\mu(y,s_0)|+C\veps \le \frac12+C\veps < \frac{44}{75}
	\end{equation*}
	for sufficiently small $\ve>0$. 
	Here, we used \eqref{4.3a2-p}, that is equivalent to 
	\begin{equation*}
		\limsup_{|y|\rightarrow \infty}|\mu(y,s_0)|\leq 1/2.
	\end{equation*}
This completes the proof. 
\end{proof}

We finally close \eqref{Utildey_M-p} in the following lemma.
\begin{lemma}\label{mainprop_1-p}
Under the  assumptions of Proposition~\ref{mainprop}, we have
	\begin{equation*}
		|(y^{2/3}+8)(U_y-\overline{U}')| \le \frac{24}{25}
	\end{equation*}
	for all $y\in\mathbb{R}$ and $s\in[s_0,\sigma_1]$.
\end{lemma}

\begin{proof}
Let $\nu (y,s) :=(y^{2/3}+8)\wt{U}_y(y,s)$, where $\wt{U}=U-\overline{U}$. 
From \eqref{EP2_2D1-p}, we see that $\nu$ satisfies  
\begin{equation}\label{nu_eq1-p}
	\partial_s \nu + D(y,s)\nu+\mathcal{U}(y,s)\nu_y = F_1(y,s) +F_2(y,s) +\int_{\mathbb{R}}{\nu(y',s) K(y,y') \,dy'},
\end{equation}
where $\mathcal{U}$ is defined in \eqref{U-p}, and
\begin{subequations}
	\begin{align*}
		D(y,s) &:= 1+\wt{U}_y+2\overline{U}'-\frac{2y^{2/3}}{3(y^{2/3}+8)}\left(\frac{3}{2}+\frac{\wt{U}+\overline{U}}{y}\right),
		\\
		F_1(y,s) &:= -\left(\frac{e^{s/2}(\kappa-\dot{\xi})}{1-\dot{\tau}}+\frac{\dot{\tau}\overline{U}}{1-\dot{\tau}}\right)(y^{2/3}+8)\overline{U}''-\frac{\dot{\tau}}{1-\dot{\tau}}(y^{2/3}+8)\overline{U}'^2,
		\\
		F_2(y,s) &:=-\frac{e^s\Phi_{yy}}{1-\dot{\tau}}(y^{2/3}+8) +\left(-\frac{\dot{\tau}(\wt{U}_y+2\overline{U}')}{1-\dot{\tau}}+\frac{2}{3}\frac{y^{2/3}}{y^{2/3}+8}\left(\frac{\dot{\tau}}{1-\dot{\tau}}\frac{U}{y}+\frac{1}{y}\frac{e^{s/2}(\kappa-\dot{\xi})}{1-\dot{\tau}}\right)\right)\nu,
		\\
		K(y,s;y') &:= -\frac{1}{1-\dot{\tau}}(y^{2/3}+8)\overline{U}''(y)\mathbb{I}_{[0,y]}(y')\frac{1}{(y')^{2/3}+8}.
	\end{align*}
\end{subequations}
It is straightforward to check by \eqref{Utildey_M-p} that 
\begin{equation*}
	\begin{split}
		D(y,s)
		&\geq 1-\frac{1}{y^{2/3}+8}+2\overline{U}'-\frac{2y^{2/3}}{3(y^{2/3}+8)}\left(\frac{3}{2}+\frac{\overline U}{y}+\frac{1}{y}\int_{0}^{y}{\frac{dy'}{{y'}^{2/3}+8}}\right) 
		=: D_{-}(y),
	\end{split}
\end{equation*}
and by \eqref{d-t-e-p} that
\begin{equation*}
	\int_{\mathbb{R}}{|K(y,s;y')|\,dy'}\leq(1+O(\veps))(y^{2/3}+8)|\overline{U}''|\int^{|y|}_{0}{\frac{dy'}{{y'}^{2/3}+8}}=: K_{+}(y).
\end{equation*}
Then, for sufficiently small $\varepsilon>0$,     we have  from \eqref{0605_m_7} that 
\begin{equation}\label{K+D-}
	K_{+}(y)\leq D_{-}(y) \qquad \text{for }|y|\geq 3, 
\end{equation}
which implies \begin{equation*}
	D(y,s)\geq  \int_{\mathbb{R}}|K(y,s;y')|\,dy' \qquad \text{for }|y|\geq 3.
\end{equation*}

Next, we will check that 
\begin{equation}\label{F-nu-p}
	\|F_1(\cdot,s)\|_{L^{\infty}(|y|\geq 3)} + \|F_2(\cdot,s)\|_{L^{\infty}(|y|\geq 3)} \leq Ce^{-s}.
\end{equation}
Thanks to \eqref{y-w-y} and \eqref{Wyybar_bdd}, we find that  $\overline{U}=O(|y|^{1/3})$, $\overline{U}'=O(|y|^{-2/3})$ and $\overline{U}''=O(|y|^{-5/3})$ as $|y|\rightarrow \infty$. 
Using \eqref{tau_dec} and \eqref{temp_1-p}, we have for $|y|\geq 3$ that
\begin{equation}\label{F1-p}
	\|F_1(\cdot,s)\|_{L^{\infty}(|y|\geq 3)}\leq Ce^{-s}.
\end{equation}
Similarly as \eqref{Fmu}, using $|\wt{U}_y|\leq 1$, we have
\begin{equation}\label{F2-p}
	\begin{split}
		\|F_2(y,s)\|_{L^{\infty}(|y|\geq 3)} &\leq C(e^{-s}+e^{-s}\|\nu\|_{L^{\infty}})
		\leq Ce^{-s},
	\end{split}
\end{equation}
where we have used \eqref{Utildey_M-p} for the last inequality. We see that \eqref{F1-p} and \eqref{F2-p} give \eqref{F-nu-p}.


In order to finish the proof, we claim that
\begin{equation}\label{claim_reg-p}
	\|\nu(\cdot,s)\|_{L^{\infty}}< \frac{24}{25}, \qquad s\in[s_0,\sigma_1].
\end{equation}
Suppose to the contrary that \eqref{claim_reg-p} fails.
First we note 
from 
\eqref{4.3a-p} that 
\begin{equation*}
	\|\nu(\cdot,s_0)\|_{L^{\infty}}\leq \frac{22}{25}.
\end{equation*}
 Then, thanks to the fact that
  $\nu \in C([s_0, \sigma_1]; L^\infty(\mathbb{R}))$,  $s_2:= \min \{ s \in [s_0, \sigma_1]: \|\nu(\cdot,s)\|_{L^{\infty}} = 24/25 \}$ is well-defined, and there exists $s_1\in (s_0, s_2)$  such that
\begin{equation}\label{34-p}
	\frac{23}{25} = \|\nu(\cdot,s_1)\|_{L^{\infty}}  \le \|\nu(\cdot,s)\|_{L^{\infty}} < \frac{24}{25} \qquad \text{for }s\in (s_1,s_2). 
\end{equation}
Then, for each $s\in[s_1,s_2]$, thanks to the smoothness of $\nu$ and its decay property \eqref{Ut_dec_p},  we see that  $C^{\nu}:=\{ y \in \mathbb{R} :  \partial_y \nu(y,s) =0 \}$  is a non-empty closed set of the critical points of $\nu$,  and  that $\| \nu(\cdot, s) \|_{L^\infty} = \max\{ | \nu(y, s) | : y \in C^{\nu}\}.$ Thus, for each $s \in [s_1,s_2]$,  one can choose a point $y_*(s)\in C^\nu$ such that 
\begin{equation}\label{claim2_nudec0-p}
	\|\nu(\cdot, s) \|_{L^{\infty}} = |\nu(y_*(s), s)|.
\end{equation}
In view of \eqref{34-p},  we see that $\nu(y_*(s), s) \neq 0$ for all $s \in [s_1,s_2]$. This implies that we can apply the same method to \eqref{AP_R1-p}. This helps to define that $\nu$ is differentiable with $y$. Moreover, due to \eqref{claim2_nudec0-p}, $ \partial_y \nu(y_\ast(s),s) =0 $. 
 We also note that 
 \begin{equation*}
 	\|\nu(\cdot,s)\|_{L^{\infty}(|y|\leq 3)}< \frac{23}{25},
 \end{equation*}
thanks to \eqref{EP2_1D1-p}. 
This together with \eqref{34-p} implies that   
\begin{equation}\label{y*'}
	|y_*(s)|\geq 3.
\end{equation} 

If $\nu(y_*(s),s)\geq 0$, then we have by \eqref{K+D-} and \eqref{y*'} that
\begin{equation}
	\begin{split}\label{nu_DK-p}
		D^\nu(y_*(s),s)\nu(y_*(s),s)&\geq D_{-}(y_*(s))\|\nu(\cdot,s)\|_{L^{\infty}}
		\\
		&\geq K_{+}(y_*(s))\|\nu(\cdot,s)\|_{L^{\infty}}\geq \left|\int_{\mathbb{R}}K^\nu(y_*(s),s;y')\nu(y',s)\,dy'\right|.
	\end{split}
\end{equation}
Similarly, if  $\nu(y_*(s),s)\leq 0$, we have 
\begin{equation}
	\begin{split}\label{nu_DK2-p}
		D^\nu(y_*(s),s)\nu(y_*(s),s)&\le - D_{-}(y_*(s))\|\nu(\cdot,s)\|_{L^{\infty}}
		\\
		&\le - K_{+}(y_*(s))\|\nu(\cdot,s)\|_{L^{\infty}} \le - \left|\int_{\mathbb{R}}K^\nu(y_*(s),s;y')\nu(y',s)\,dy'\right|.
	\end{split}
\end{equation}
Using  \eqref{F-nu-p}, \eqref{nu_DK-p}, \eqref{nu_DK2-p} and the property $\partial_y \nu(y_*(s), s) =0$ which is due to the fact that $y_*(s)\in C^{\nu}$, we have from \eqref{nu_eq1-p} that if   $\nu(y_*(s),s)> 0$, 
\begin{equation}\label{nu_temp-p}
	\begin{split}
		\partial_s \nu(y_*(s) ,s) & \leq  Ce^{-s}+\int_{\mathbb{R}} \nu(y',s)K^\nu(y_*,s;y')\,dy'-D^\nu(y_*,s)\nu
		\leq Ce^{-s},
	\end{split}
\end{equation}
and symmetrically that if $\nu(y_*(s),s)< 0$, 
\begin{equation}\label{nu_temp--p}
	\begin{split}
		\partial_s \nu(y_*(s) ,s) & \ge - Ce^{-s}.
	\end{split}
\end{equation}

For fixed $s$, by the definition of $y_*$, it holds  that \[ \|\nu(\cdot,s-h)\|_{L^\infty} \geq |\nu(y_*(s) - h \mathcal{U} (y_*(s),s), s - h)| 
\] 
for any small $h>0$. 
Then, it is straightforward to check that 
\begin{equation}\label{AP_R1} 
\begin{split}
\lim_{h \to 0^+} \frac{\|\nu(\cdot,s-h)\|_{L^\infty} - \|\nu(\cdot,s)\|_{L^\infty}}{-h} 
& \leq (\partial_s + \mathcal{U} (y_*(s),s)\partial_y)|\nu|(y,s)|_{y=y_*(s)}, 
\end{split}
\end{equation} 
provided that the limit on the LHS of \eqref{AP_R1} exists. Note that by Rademacher's theorem, 
$ \| \nu(\cdot, s)\|_{L^{\infty}}$, being Lipschitz continuous in $s$, is differentiable at almost all $s\in[s_1, s_2]$. Thus, since  $\partial_y \nu(y_*(s), s) =0$, we deduce from \eqref{AP_R1} that 
\begin{equation} \label{L-thm-p}
\begin{split}
\frac{d}{ds} \| \nu(\cdot, s)\|_{L^{\infty}} 
& \leq  \left\{ \begin{array}{l l}
\partial_s \nu(y, s)|_{y=y_*(s)} & \text{if } \nu(y_*(s),s)>0, \\
-\partial_s \nu(y, s)|_{y=y_*(s)} & \text{if } \nu(y_*(s),s)<0
\end{array}
\right.  
\end{split}
\end{equation} 
for almost all $s\in[s_1, s_2]$. 

%
%
%
%
%
%

Thus, using \eqref{L-thm-p} with  \eqref{nu_temp-p} and \eqref{nu_temp--p}, we have 
\begin{equation*}
\begin{split} 
 \| \nu(\cdot, s_2) \|_{L^{\infty}} 
&=  \| \nu(\cdot, s_1) \|_{L^{\infty}} +  \int_{s_1}^{s_2} \frac{d}{ds} \| \nu(\cdot, s)\|_{L^{\infty}} \; ds  \\ 
& \le \| \nu(\cdot, s_1) \|_{L^{\infty}} + \int_{s_1}^{s_2} C e^{-s} \; ds \\
&  \leq  \| \nu(\cdot, s_1) \|_{L^{\infty}} + C \veps.
\end{split} 
\end{equation*}
This together with \eqref{34-p} leads to a contradiction  for sufficiently small $\ve>0$,  which proves \eqref{claim_reg-p}, in turn Lemma~\ref{mainprop_1-p}. 
\end{proof}

%

\section{Pressureless Euler system} \label{PE-blow-up}
In this section,  we present a similar blow-up result for the pressureless Euler system: 
\begin{subequations}\label{PE}
	\begin{align}
		& n_t +  (n v)_x = 0, \label{PE_1} 
		\\ 
		& v_t  + v v_x = 0, \quad x\in \mathbb{R}, t\ge-1. 
	\end{align}
\end{subequations} 
We consider the initial value problem for \eqref{PE} with the initial data $(n,v)(x,-1) = (n_0, v_0)(x)$.
		\begin{theorem}\label{Thm-PE}
	Suppose that  $(n_0-1, v_0)\in  C^3(\mathbb{R})\times C^4(\mathbb{R}) $ and   that $\dd_x v_0(x)$ has a non-degenerate global minimum at some point $x_0\in\mathbb{R}$ with $\dd_x v_0(x_0)<0$.  
	Then, there is a unique smooth solution $(n,u)\in ( C^3\times C^4)  \left([-1,T_\ast) \times \mathbb{R}\right)$ to \eqref{PE} whose the maximal existence time $T_\ast = |\inf_x \partial_x v_0(x)|^{-1}$ is finite.  Furthermore, it holds that
	\begin{enumerate}[(i)]
		\item $\sup_{t<T_\ast} \left[ v(\cdot, t) \right]_{C^{\beta}}<\infty \quad \text{for}\quad \beta\leq 1/3$;
		\item for $\beta>1/3$, $ \lim_{t\rightarrow T_\ast} \left[ v (\cdot, t) \right]_{C^{\beta}}=\infty$
		with the temporal blow-up rate  
		\ \begin{equation*} \left[ v(\cdot, t) \right]_{C^\beta}\sim (T_\ast-t)^{-\frac{3\beta-1}{2}}
		\end{equation*} for all $t$ sufficiently close to $T_\ast$;
		\item 
the density blows up as 
		\begin{equation*}
			n(x,t)  \sim \frac{1}{(T_\ast-t) + (x-x_*)^{2/3}}
		\end{equation*} 
		for all $(x,t)$ sufficiently close to $(x_*,T_\ast)$, where $x_\ast$ is the location of blow-up. 
	\end{enumerate} 
\end{theorem}
\subsection{Outline of the proof}
Since the proof of Theorem~\ref{Thm-PE} is simpler than that of Theorem~\ref{mainthm1}, we briefly outline the proof here. 
By a standard characteristic method, one can show that 
\[ \inf_{x\in\mathbb{R} }  \partial_x v (x,t) \searrow -\infty, \quad \sup_{x\in \mathbb{R}}  n(x,t) \nearrow \infty \quad \text{ as } t\nearrow T_\ast. \] 
By this together with   the scaling invariances of the Burgers equation, 
we can  assume without loss of generality that   
\begin{equation}
		\inf_{x\in \mathbb{R}} \partial_x v_0(x) = \partial_x v_0(0)=-1, \quad  \partial_x^2 v_0(0) =0, \quad \partial_x^3 v_0(0) =6 >0, 
		\end{equation} 
		and  we can further assume  that  $\partial_x v_0$ and $\partial_x n_0$  satisfy
\begin{equation} \label{asymp-cond}
	\left| \partial_x v_0 (x)-\overline{U}' ( {x} )\right|\leq \min\left\{\frac{ {x}^2}{40\left(1+ {x}^2\right)}, \frac{22}{25\left(8+ {x}^{2/3}\right)}\right\}. 
\end{equation}

By introducing the self-similar variables, 
\begin{equation*}\label{change_var-pe}
	y\left(x,t\right) := \frac{x}{\left( -t\right)^{3/2}}, \quad s(t) :=-\log\left(  - t \right),
\end{equation*}		
and new unknowns $N$ and $V$ as 
\[ n(x,t) = e^s N(y,s), \quad v(x,t) = e^{-s/2} V(y,s),  \] 
we see that $N$ and $V$  satisfy the equations 
\begin{subequations}\label{PE-new-var}
	\begin{align}
		& \partial_s N + \left(1 + {V_y}\right)N + \mathcal{V} N_y = - {e^{-s}V_y} , \label{PE_1-p} 
		\\ 
		& \partial_s V - \frac{1}{2}V + \mathcal{V} V_y =  0, \label{PE_2-p} 
	\end{align}
\end{subequations}
where 
\begin{equation}\label{V-ep}
	\mathcal{V}:=  {V}  + \frac{3}{2}y.
\end{equation}

%
Loosely following the analysis carried out for \eqref{EP-p} through the previous sections, with no modulations nor the terms arising from the Poisson equation, we can establish  the global uniform estimates: 
\[ | V_y (y,s) |\lesssim y^{-2/3}, \quad (y^{2/3} +1) N(y,s) \sim 1  \quad \text{ for all } y \in \mathbb{R}, s\ge0.\] 
These estimates imply that 
 $[v(\cdot, t) ]_{C^{\beta}(\mathbb{R}) } \lesssim 1$ if $\beta\le1/3$. 
 Furthermore, if $\beta>1/3$, it holds that $[v(\cdot, t) ]_{C^{\beta}(\mathbb{R}) } \to \infty$ as $t\nearrow0$ and 
\begin{equation*}
		 | n(x,t) |\sim ( x^{2/3} -t )^{-1} . 
	\end{equation*}  

\section{Appendix}

\subsection{Rescaling initial data}\label{scaling-IC} 
In this subsection, we discuss how the initial condition \eqref{init_w_3-p} can be relaxed. 
Thanks to the time and space translation invariances, one can assume that 
the initial time $t=-\ve$ and $-\partial_x u(-\ve, x)$ attains its local maximum at $x=0$. 
Furthermore, using the scaled variables, $t'= t$, $x'=\mu x$,   and new unknown functions  
\begin{equation*}
	\rho(t,x)=\frac{1}{\mu}\wt{\rho}(t',x'), \quad u(t,x)=\frac{1}{\mu}\wt{u}(t',x'),  \quad \phi(t,x)=\wt{\phi}(t',x'), 
\end{equation*}
we see from \eqref{EP-p}  that 
\begin{equation}\label{wt_eq}
	\begin{split}
		&\wt{\rho}_{t'}+(\wt{\rho}\wt{u})_{x'}=0,
		\\
		&\wt{u}_{t'}+\wt{u}\wt{u}_{x'}=-\mu^2 \wt{\phi}_{x'},
		\\
		&-\mu^2\wt{\phi}_{x'x'}=\mu^{-1}\wt{\rho}-e^{\wt{\phi}}.
	\end{split}
\end{equation}
By choosing $\mu$ appropriately, one can make $\partial_x u_0(0)=\partial_{x'}\wt{u}_0(0)$ and  $\partial_{x'}^3\wt{u}_0(0)= \mu^2\partial_x^3u_0(0)= 6\veps^{-4}$ as desired.  
We note that the coefficient $\mu$  appears in  the system \eqref{wt_eq}, and a similar analysis can be carried out as long as  $\mu$ is independent of $\veps$. 

%

\subsection{Useful inequalities and their  verification}	\label{inequalities}	

In this subsection,  we give a set of   the inequalities that are used in the course of our analysis. 

\begin{lemma}[Bound of $\overline{U}'$ for small $|y|$] It holds that for all $y\in\mathbb{R}$,
	\begin{equation} \label{y-w-y2}
		-\frac{1}{1+\frac{3y^2}{(3y^2+1)^{2/3}}}\leq  \overline{U}' (y) \leq 0.
	\end{equation}
\end{lemma}
The proof is given in Subsection 6.2 of the Appendix in \cite{BKK}. 

\begin{lemma}[Decay rate for $\overline{U}'$ and $\overline{U}''$] There exists a constant $C>0$ such that for all $y\in\mathbb{R}$ the following hold:
	\begin{subequations}
		\begin{align}
			&|\overline{U}'(y)|\leq  C(1+y^2)^{-1/3}, \label{y-w-y} \\
			& |\overline{U}''(y)|\leq C(1+y^2)^{-5/6}. \label{Wyybar_bdd}
		\end{align}
	\end{subequations}
\end{lemma}
For the proof, we refer to the proof of Lemma 6.1 in \cite{BKK}. 
We also present some useful inequalities involving $\overline{U}$. 
\begin{lemma}\label{4.3}
	For some $\lambda>1$, it holds that for all $y\in\mathbb{R}$,
	\begin{subequations}
		\begin{align}
			&\frac{y^2}{5(1+y^2)}+\frac{16 y^2}{5(1+8y^2)} \leq 1+2\overline{U}'+\frac{2}{1+y^2}\left(\frac{3}{2}+\frac{\overline U}{y}\right),  \label{0605_m_3}
			\\
			&\lambda |\overline{U}''|  \frac{ 1+y^2}{y^2} \int^{|y|}_0{\frac{(y')^2}{1+(y')^2 }\,dy'}\leq\frac{3y^2}{1+8y^2}+\frac{y^2}{30(1+y^2)}. \label{0605_m_4}
		\end{align}
	\end{subequations} 
	 Furthermore, for $|y|\geq 3$, we have
	\begin{equation}
			\lambda |\overline{U}''| (y^{2/3}+8) \int^{|y|}_{0}{\frac{dy'}{{y'}^{2/3}+8}}\leq
			1-\frac{1}{y^{2/3}+8}+2\overline{U}'-\frac{2y^{2/3}}{3(y^{2/3}+8)}\left(\frac{3}{2}+\frac{\overline{U}}{y}+\frac{1}{y}\int^y_0 \frac{dy'}{y'^{2/3}+8}\right).\label{0605_m_7}
	\end{equation}
\end{lemma}
For the  proofs, we refer to those of  Lemma~6.2 and Lemma~6.3 in Appendix of \cite{BKK}.

\subsection{Poisson equation}

\begin{lemma}\label{Lem_Pois}
	Let $f(y,s)$ be a continuous function such that
	\begin{equation}\label{09}
		\sup_{(y,s)\in\mathbb{R}\times [0,\infty)}\; y^{2/3} |f(y,s)| \leq C_f
	\end{equation}
	for some constant $C_f>0$. Suppose that $C_f>0$ is a constant such that 
	\begin{subequations}
		\begin{align}
			& C_f\sup_{y\in \mathbb{R}}I \leq \frac14, \label{B_CCon1} 
			\\ 
			& e^{1/4}C_f\left(\sup_{y\in \mathbb{R}}I\right)^2 \leq 2, \label{B_CCon2}
		\end{align}
	\end{subequations}
	  where $I(y)$ is defined in \eqref{Aeq}. Then, the solution to the Poisson equation
	\begin{equation}\label{B_Pois}
		- \partial_{yy}\Phi = f(y,s) + 1 - e^\Phi
	\end{equation}
	satisfies for   $n=0,1,2$,
	\begin{equation}\label{B_Pois2} 
		\sup_{(y,s)\in\mathbb{R}\times[0,\infty)}|(y^{2/3} + 1)\partial_y^n\Phi(y,s)| \leq C_f.
	\end{equation}
\end{lemma}

\begin{proof}
	
	We define a sequence of functions as follows:
	\[
	\Phi_{n+1}   = \frac{1}{2} \int_{-\infty}^\infty e^{-|y-y'|}\left(f - \left(e^{\Phi_{n}} - 1 - \Phi_{n}\right)\right)(y',s)\,dy' \quad \text{for } n \in \mathbb{N},
	\]
	where $\Phi_1$ is the solution to the linear inhomogeneous equation $(1 - \partial_{yy})\Phi_1 = f$. By induction, we claim that 
	\begin{equation}\label{allPhi_14}
		\|(y^{2/3}+1)\Phi_n(y,s)\|_{L^{\infty}}\leq C_f\sup_{y\in \mathbb{R}}I(y)  \quad \text{ for all } n\in\mathbb{N}.
	\end{equation} 
	For $n=1$, \eqref{allPhi_14} follows from 
	\begin{equation}\label{BPois1}
		\begin{split}
			|(y^{2/3}+1)\Phi_1|
			& = \left| \frac{(y^{2/3}+1)}{2}\int_{-\infty}^\infty e^{-|y-y'|} f(y',s) \,dy' \right| 
			 \leq  \frac{C_f}{2} I(y)
		\end{split}
	\end{equation}
	thanks to  \eqref{09}.  Assume that \eqref{allPhi_14} holds for $n=k\ge1$.  Then, using \eqref{B_CCon1}, \eqref{allPhi_14} and \eqref{BPois1}, we have
	\begin{equation}\label{BPoisn}
		\begin{split}
			(y^{2/3}+1)|\Phi_{k+1}(y,s)|
			& = (y^{2/3}+1)\left|\Phi_1(y,s)-\frac{1}{2}\int^{\infty}_{-\infty}e^{-|y-y'|}(e^{\Phi_k}-1-\Phi_k)\,dy' \right| 
			\\
			& \leq (y^{2/3}+1) \left(|\Phi_1| + \frac{1}{4} \int^{\infty}_{-\infty}e^{-|y-y'|}|\Phi_k|^2 e^{|\Phi_k|}\,dy' \right)   
			\\
			& \leq \frac{C_f}{2}\sup_{y\in \mathbb{R}}I  + \frac{e^{1/4}}{4}\big(C_f\sup_{y\in \mathbb{R}}I\big)^2 (y^{2/3}+1) \int^{\infty}_{-\infty}e^{-|y-y'|}(y'^{2/3}+1)^{-2}\,dy' 
			\\ 
			& \leq C_f\sup_{y\in \mathbb{R}}I \left( \frac{1}{2} + \frac{e^{1/4}}{4}C_f\big(\sup_{y\in \mathbb{R}}I\big)^2 \right),
		\end{split}
	\end{equation}
	where we have used
	\[
	|e^{\Phi_n} - 1 - \Phi_n| \leq  \frac{|\Phi_n|^2}{2}\sum_{j=2}^\infty \frac{|\Phi_n|^{j-2}}{j!/2} \leq \frac{|\Phi_n|^2}{2}e^{|\Phi_n|}
	\]
	and $|\Phi_n|\leq 1/4$ by \eqref{B_CCon1} and \eqref{allPhi_14} for the second inequality. Together with \eqref{B_CCon2}, \eqref{BPoisn} implies that \eqref{allPhi_14} holds true for $n=k+1$. 
	
	Now we show that $\{\Phi_n\}$ is Cauchy in $C_b(\mathbb{R}\times [0,\infty))$.  Since  $|\Phi_n| \leq 1/4$ by \eqref{B_CCon1} and \eqref{allPhi_14}, it holds that
	\begin{equation*}
		2(1-e^{-1/2}) e^{-1/4}\leq 2(1-e^{-1/2}) e^{\Phi_n} \leq \frac{1-e^{-(\Phi_n-\Phi_{n-1})}}{\Phi_n-\Phi_{n-1}}e^{\Phi_n}\leq  -2(1-e^{1/2})e^{\Phi_n} \leq  -2(1-e^{1/2})e^{1/4}
	\end{equation*}
	since $x \mapsto  (1-e^{-x})/{x}$ is a positive decreasing function on $\mathbb{R}$. Hence, we see that
	\begin{equation}\label{nu_n+11}
		\begin{split}
			\left| \frac{(e^{\Phi_n}-\Phi_n)-(e^{\Phi_{n-1}}-\Phi_{n-1})}{\Phi_n-\Phi_{n-1}}  \right|
			& = \left|1-e^{\Phi_n}\frac{1-e^{-(\Phi_n-\Phi_{n-1})}}{\Phi_n-\Phi_{n-1}} \right| 
			\\
			& \leq \max\left\{ 1-2e^{-1/4}(1-e^{-1/2}), -2e^{1/4}(1-e^{1/2})-1\right\} 
			\\
			& =:c_1 < 1.
		\end{split}
	\end{equation}
	Since $c_1<1$ and 
	\begin{equation}\label{nu_n+1}
		\begin{split}
			|\Phi_{n+1}-\Phi_n|
			& \leq \frac{1}{2}\int^{\infty}_{-\infty}e^{-|y-y'|}|(e^{\Phi_n}-\Phi_n)-(e^{\Phi_{n-1}}-\Phi_{n-1})|\,dy' 
			\\ 
			& \leq  c_1\sup_{(y,s)\in\mathbb{R}\times[0,\infty)}|\Phi_{n}-\Phi_{n-1}|,
		\end{split}
	\end{equation}  
	we see  that $\{\Phi_n\}$ be a Cauchy sequence in $C_b(\mathbb{R}\times[0,\infty))$. Therefore,  $\lim_{n\rightarrow \infty}\Phi_n=\Phi$ exists uniformly in $(y,s)$. Furthermore, using \eqref{nu_n+11}, we obtain that
	\begin{equation}\label{BPois2}
		\begin{split}
			\Phi 
			& =\lim_{n\rightarrow \infty}\Phi_{n+1}
			\\
			& =\lim_{n\rightarrow \infty} \frac{1}{2}\int^{\infty}_{-\infty}e^{-|y-y'|}( f- e^{\Phi_n}+1+\Phi_n)(s,y')\,dy' 
			\\
			&=\frac{1}{2}\int^{\infty}_{-\infty}e^{-|y-y'|}( f- e^{\Phi}+1+\Phi)(s,y')\,dy',
		\end{split}
	\end{equation}
	which means that $\Phi$ is the (unique) solution of the Poisson equation.
	
	By taking the limit of \eqref{allPhi_14}, we see that \eqref{B_Pois2} holds true for $n=0$. By taking the derivatives of \eqref{BPois2}, we get
	\begin{subequations}\label{B_Pois3}
		\begin{align}
			\partial_y \Phi  & =  \frac{1}{2}\int_{-\infty}^\infty -\frac{y-y'}{|y-y'|} e^{-|y-y'|}  (f - e^\Phi + 1 + \Phi)\,dy', 
			\\
			\partial_{yy} \Phi  & =  \frac{1}{2}\int_{-\infty}^\infty  e^{-|y-y'|}  (f - e^\Phi + 1 + \Phi)\,dy'  - (f - e^\Phi + 1  + \Phi).
		\end{align}
	\end{subequations} 
	By applying \eqref{B_Pois2} for $n=0$ to \eqref{B_Pois3}, it is straightforward to see  that \eqref{B_Pois2} holds for $n=1,2$. 
\end{proof}

\subsection{Maximum Principles}
We present a maximum principle,  which is a modified form of the one developed in  \cite{BSV}, to apply to our analysis.
We consider the initial value problem: 
\begin{equation}\label{IVP-f}
	\begin{split} 
		& \partial_s f(y,s)+D(y,s) f(y,s)+U(y,s)\partial_yf(y,s)=F(y,s)+\int_{\mathbb{R}}f(y',s)K(y,s;y')\,dy',
		\\
		& f(y,s_0) = f_0(y), \quad 
		s\in[s_0,\infty), \quad y\in \mathbb{R}. 
	\end{split} 
\end{equation}

	\begin{lemma}(\cite{BKK})
		\label{max_2}
		Let $f$ be a classical solution to IVP \eqref{IVP-f}. Let $\Omega \subseteq \mathbb{R}$ be any compact set. 
		Suppose that the following hold:  
		\begin{subequations}
			\begin{align}
				& \|f(\cdot,s)\|_{L^{\infty}(\Omega)}\leq m_0, \label{max_2_1}
				\\
				&  \|f(\cdot,s_0)\|_{L^{\infty}(\mathbb{R})}\leq m_0, \label{max_2_1'}
				\\
				& \int_{\mathbb{R}}|K(y,s;y')|\,dy'\leq \delta D(y,s)\quad \text{for}\quad (y,s)\in \Omega^c\times [s_0,\infty), \label{max_2_2}
				\\
				& \inf_{(y,s)\in \Omega^c\times [s_0,\infty)}D(y,s)\geq \lambda_D >0, \label{max_2_3}
				\\
				& \|F(\cdot,s)\|_{L^{\infty}(\Omega^c)}\leq F_0, \label{max_2_4}
				\\ 
				& \limsup_{|y|\rightarrow \infty}|f(y,s)| <  2m_0 \label{max_2_6}
			\end{align}
		\end{subequations} 
		for some $m_0, F_0, \lambda_D >0$ and $\delta<1$. If $m_0\lambda_D> F_0/(2-2\delta)$, 
		then $\|f(\cdot,s)\|_{L^{\infty}(\mathbb{R})}\leq 2m_0$.
	\end{lemma}
	For the proof, we refer to that of Lemma 6.6 in Appendix of \cite{BKK}. 
	
\subsection{Blow-up criterion}\label{BC-app}
Owing to \eqref{Phi_0_M}, there exist constants $m_1$ and $m_2$ (depending only on the initial energy $H(0)$) such that 
\begin{equation}\label{A3}
0 < m_1 \leq  e^\phi \leq m_2
\end{equation}
 as long as the smooth solution to \eqref{EP-p} exists. In this subsection, we show that if 
 \begin{equation}\label{A1}
\partial_x u_0(\alpha) \geq \textstyle\sqrt{2\rho_0(\alpha)-m_2+\frac{4\rho_0^2(\alpha)}{m_1}\left( \frac{m_2-m_1}{m_1} \right)} \quad \text{for some } \alpha\in\mathbb{R},
\end{equation}
or
 \begin{equation}\label{A2}
\partial_x u_0(\alpha) \leq -\sqrt{2\rho_0(\alpha)-m_1} \quad \text{for some } \alpha\in\mathbb{R},
\end{equation}
then the maximal existence time $T_\ast$ of the $C^1$ solution to \eqref{EP-p} is finite. 

Following the approach of \cite{BCK}, we derive the associated second order ODE. Then, we make use of the variation of constants formula  and conduct a direct comparison with the associated differential inequalities obtained from \eqref{A3}. The proof given below is simple and clearly indicates that one cannot   obtain a sharper blow-up criterion than \eqref{A1}--\eqref{A2} when relying solely on \eqref{A3} and the ODE comparison. We remark that  \eqref{A1} and \eqref{A2} are the same as the criterion given in \cite{CKKT} (Theorem 1.11), where the authors consider the first order system of   ODEs and employ the comparison argument using  \eqref{A3}.

Let us define the characteristic curve $x(\alpha,t)$ along $u \in C^1$ as follows:
\begin{equation}\label{CharODE}
\dot{x} = u(x(\alpha,t),t), \quad x(\alpha,0)=\alpha \in \mathbb{R}, \quad t \geq 0,
\end{equation}
where $\dot{•}:=d/dt$. Here, the initial position $\alpha$ is considered as a parameter. By differentiating \eqref{CharODE} in $\alpha$,  we obtain that
\begin{equation}\label{var_ODE}
\dot{w} = u_x(x(\alpha,t),t) w, \quad w (\alpha,0)=1, \quad t \geq 0,
\end{equation}
where $w=w(\alpha,t):= \textstyle{\frac{\partial x}{\partial \alpha}(\alpha,t)}$. Using \eqref{EP-p}, we see that  $w$ satisfies
\begin{equation}\label{wrho}
\rho (x(\alpha,t),t) w(\alpha,t)  = \rho_0(\alpha) 
\end{equation}
and 
\begin{equation}\label{2ndOrdODE}
\ddot{w} + e^{\phi(x(\alpha,t),t)} w = \rho_0(\alpha), \quad w(\alpha,0)=1 , \quad \dot{w}(\alpha,0) = u_{0x}(\alpha),
\end{equation} 
and furthermore, the following statements hold:
 \begin{enumerate}
 \item for each $\alpha \in \mathbb{R}$, 
\begin{equation*}
\left.
\begin{array}{lll}
\lim_{t \nearrow T_*} w(\alpha,t)=0 \quad  & \text{iff} \quad \lim_{t \nearrow T_*} \rho(x(\alpha,t),t) = +\infty &  \text{iff}  \quad \liminf_{t \nearrow T_\ast } u_x\left(x(\alpha,t),t \right) = -\infty,
\end{array} 
\right.
\end{equation*} 
\item if $\lim_{t \nearrow T_*} w(\alpha,t)=0$ holds for some  $\alpha\in\mathbb{R}$, then $u_x\left(x(\alpha,t),t \right) \approx (t-T_\ast)^{-1}$  for all $t<T_\ast$ sufficiently close to $T_\ast$.
\end{enumerate}
(See \cite{BCK},  pp.15-16.) Therefore, the problem boils down to finding the zeros of $w$ for $t\ge0$.

From \eqref{A3} and \eqref{2ndOrdODE}, we have $\ddot{w} + m_1 w - \rho_0  =: h_1 \leq 0$ and $\ddot{w} + m_2 w - \rho_0  =: h_2 \geq  0$.  Using the variation of constants formula and recalling $w(\alpha,0)=1$ (see \eqref{var_ODE}), we obtain that for $k=1,2$,
\begin{equation}\label{Appn1}
w(t) = \left( 1 - \frac{b}{m_k} \right)\cos(\sqrt{m_k} t) + \dot{w}_0 \frac{\sin(\sqrt{m_k} t)}{\sqrt{m_k}} + \frac{b}{m_k} +  \frac{1}{\sqrt{m_k}}\int_0^t \sin(\sqrt{m_k} t') h_k(t-t')\,dt', 
\end{equation}
where we let $w(t)=w(\alpha,t)$, $\dot{w}_0 =\dot{w}(\alpha,0)$, and $b=\rho_0(\alpha)$ for simplicity. We notice that  the integrand of \eqref{Appn1} with $k=1$ is nonpositive  on $t\in[0,\textstyle{\frac{\pi}{\sqrt{m_1}}}]$. Similarly, for $k=2$, it is  nonnegative on $t\in[0,\textstyle{\frac{\pi}{\sqrt{m_2}}}]$. 

We first present the result of \cite{BCK} (see Theorem 1.2 and Remark 3 of \cite{BCK}). If $2b \leq m_1$ (i.e., $2\rho_0(\alpha) \leq m_1$ for some $\alpha$), then putting $t=\textstyle{\frac{\pi}{\sqrt{m_1}}}$ into \eqref{Appn1} with $k=1$, we see that
\[
w\left(\textstyle{\frac{\pi}{\sqrt{m_1}}}\right) \leq -1+2\frac{b}{m_1} \leq 0.
\]
Since $w(0)=1$, by the intermediate value theorem, there is $T_\ast \in[0,\textstyle{\frac{\pi}{\sqrt{m_1}}}]$ such that $\textstyle\lim_{t\searrow T_\ast}w(t)=0$. 

In what follows, using this result, we show that \eqref{A1} and  \eqref{A2} are sufficient conditions for the finite time blow-up.  Since $h_1 \leq 0$ and $h_2 \geq 0$, making use of the trigonometric identity, we obtain from \eqref{Appn1} that
\begin{subequations}
			\begin{align}
				& w(t) \leq  \sqrt{\dot{w}_0^2/m_1 + \left(1-b/m_1 \right)^2}\cos(\sqrt{m_1}t - \theta) + \frac{b}{m_1} \quad \text{for } t\in[0,\textstyle{\frac{\pi}{\sqrt{m_1}}}], \label{Appn1_1}
				\\
				&  w(t) \geq  \sqrt{\dot{w}_0^2/m_2 + \left(1-b/m_2 \right)^2}\cos(\sqrt{m_2}t - \theta) + \frac{b}{m_2}  \quad \text{for } t\in[0,\textstyle{\frac{\pi}{\sqrt{m_2}}}], \label{Appn1_2}
			\end{align}
		\end{subequations} 
where $\theta$ satisfies
\[
\sin\theta = \frac{\dot{w}_0}{\sqrt{m_k}\sqrt{\dot{w}_0^2/m_k + \left(1-b/m_k \right)^2}} \quad \text{and} \quad \cos\theta = \frac{1-b/m_k}{\sqrt{\dot{w}_0^2/m_k + \left(1-b/m_k \right)^2}}.
\]

We consider the case $\dot{w}_0 > 0$. It is enough to show that $w(t_1)  \geq 2\rho_0(\alpha)/m_1$  for some $t_1>0$. Indeed, using \eqref{wrho}, we have $m_1 \geq 2\rho_0(\cdot)/w(t_1,\cdot)  = 2\rho(t_1,\cdot)$. If a smooth solution to \eqref{EP-p} exists on $[0,t_1]$, then we apply the same argument as above with $w(\alpha,0)=1$.
We have $\theta \in \textstyle [0,\pi]$ since $\dot{w}_0 > 0$ (i.e., $\partial_x u_0(\alpha)>0$  for some $\alpha$). Thus, $\sqrt{m_2}t-\theta \in \textstyle [-\pi,\pi]$ for $t\textstyle\in[0,\frac{\pi}{\sqrt{m_2}}]$. We notice that $\cos(\sqrt{m_2}t -\theta)$ achieves its maximum value at $t=\textstyle{\frac{\theta}{\sqrt{m_2}}}>0$, where the strict inequality holds due to  $\dot{w}_0>0$. Using \eqref{Appn1_2}, it is easy to check that  $\dot{w}_0 \geq \textstyle\sqrt{2b-m_2+\frac{4b^2}{m_1}\left( \frac{m_2-m_1}{m_1} \right)}$ (i.e., \eqref{A1}) implies that  there is $t_1>0$ such that  $w(t_1)  \geq 2\rho_0(\alpha)/m_1$.

Now, we consider the case $\dot{w}_0\leq 0$.  Since $\dot{w}_0\leq 0$ (i.e.,  $\partial_x u(\alpha,0) \leq 0$ for some $\alpha$), we have $\theta \in \textstyle [-\pi,0]$, and thus $\sqrt{m_1}t-\theta \in \textstyle[0,2\pi]$. It is straightforward to see from \eqref{Appn1_1} that  if $\dot{w}_0 \leq -\sqrt{2b-m_1}$ (i.e., \eqref{A2}) holds, then  $w \leq 0$ at some point on the interval $[0,\textstyle{\frac{\pi}{\sqrt{m_1}}}]$. Since $w(0)=1$, there is $T_\ast \in[0,\textstyle{\frac{\pi}{\sqrt{m_1}}}]$ such that $\textstyle\lim_{t\searrow T_\ast}w(t)=0$. We finish the proof.

\section*{Acknowledgments.}
J.B. was supported by the National Research Foundation of Korea grant funded by the Ministry of Science and ICT (NRF-2022R1C1C2005658). 
B.K. was supported by Basic Science Research Program through the National Research Foundation of Korea (NRF) funded by the Ministry of science, ICT and future planning (NRF-2020R1A2C1A01009184). 

 \end{document}